\documentclass[12pt]{amsart}
\usepackage{algorithmic}
\usepackage{algorithm}
\usepackage{epsfig}
\usepackage{amsmath}
\usepackage{amssymb}
\usepackage{amsbsy}
\usepackage{amsfonts}
\usepackage{latexsym}
\usepackage{amsopn}
\usepackage{amstext}
\usepackage{amsopn}
\usepackage{amstext}
\usepackage{amsxtra}
\usepackage{euscript}
\usepackage{amscd}
\usepackage{bm}
\usepackage{pgf}
\usepackage{tikz}
\usepackage{subfigure}

\usepackage{amsmath,amssymb,amsbsy,amsfonts,amsthm,latexsym,
        amsopn,amstext,amsxtra,euscript,amscd,mathrsfs,graphics,epsfig,color}
\usepackage{graphicx}

\usepackage{float}

\newfont{\teneufm}{eufm10}
\newfont{\seveneufm}{eufm7}
\newfont{\fiveeufm}{eufm5}
\newfam\eufmfam
  \textfont\eufmfam=\teneufm \scriptfont\eufmfam=\seveneufm
  \scriptscriptfont\eufmfam=\fiveeufm
\def \bsigma{\bm \sigma}

\def\cA{{\mathcal A}}

\def\cF{{\mathcal F}}
\def\cG{{\mathcal G}}
\def\cH{{\mathcal H}}

\def\cJ{{\mathcal J}}

\def\cV{{\mathcal V}}

\def\cX{{\mathcal X}}

\newcommand{\rad}[1]{\mathrm{rad}\, #1}

\def\mand{\qquad \text{and} \qquad}

\def\bbbf{{\rm I\!F}}

\def\bbbc{{\mathchoice {\setbox0=\hbox{$\displaystyle\rm C$}\hbox{\hbox
to0pt{\kern0.4\wd0\vrule height0.9\ht0\hss}\box0}}
{\setbox0=\hbox{$\textstyle\rm C$}\hbox{\hbox
to0pt{\kern0.4\wd0\vrule height0.9\ht0\hss}\box0}}
{\setbox0=\hbox{$\scriptstyle\rm C$}\hbox{\hbox
to0pt{\kern0.4\wd0\vrule height0.9\ht0\hss}\box0}}
{\setbox0=\hbox{$\scriptscriptstyle\rm C$}\hbox{\hbox
to0pt{\kern0.4\wd0\vrule height0.9\ht0\hss}\box0}}}}
\def\bbbq{{\mathchoice {\setbox0=\hbox{$\displaystyle\rm
Q$}\hbox{\raise
0.15\ht0\hbox to0pt{\kern0.4\wd0\vrule height0.8\ht0\hss}\box0}}
{\setbox0=\hbox{$\textstyle\rm Q$}\hbox{\raise
0.15\ht0\hbox to0pt{\kern0.4\wd0\vrule height0.8\ht0\hss}\box0}}
{\setbox0=\hbox{$\scriptstyle\rm Q$}\hbox{\raise
0.15\ht0\hbox to0pt{\kern0.4\wd0\vrule height0.7\ht0\hss}\box0}}
{\setbox0=\hbox{$\scriptscriptstyle\rm Q$}\hbox{\raise
0.15\ht0\hbox to0pt{\kern0.4\wd0\vrule height0.7\ht0\hss}\box0}}}}
\def\bbbt{{\mathchoice {\setbox0=\hbox{$\displaystyle\rm
T$}\hbox{\hbox to0pt{\kern0.3\wd0\vrule height0.9\ht0\hss}\box0}}
{\setbox0=\hbox{$\textstyle\rm T$}\hbox{\hbox
to0pt{\kern0.3\wd0\vrule height0.9\ht0\hss}\box0}}
{\setbox0=\hbox{$\scriptstyle\rm T$}\hbox{\hbox
to0pt{\kern0.3\wd0\vrule height0.9\ht0\hss}\box0}}
{\setbox0=\hbox{$\scriptscriptstyle\rm T$}\hbox{\hbox
to0pt{\kern0.3\wd0\vrule height0.9\ht0\hss}\box0}}}}
\def\bbbs{{\mathchoice
{\setbox0=\hbox{$\displaystyle     \rm S$}\hbox{\raise0.5\ht0\hbox
to0pt{\kern0.35\wd0\vrule height0.45\ht0\hss}\hbox
to0pt{\kern0.55\wd0\vrule height0.5\ht0\hss}\box0}}
{\setbox0=\hbox{$\textstyle        \rm S$}\hbox{\raise0.5\ht0\hbox
to0pt{\kern0.35\wd0\vrule height0.45\ht0\hss}\hbox
to0pt{\kern0.55\wd0\vrule height0.5\ht0\hss}\box0}}
{\setbox0=\hbox{$\scriptstyle      \rm S$}\hbox{\raise0.5\ht0\hbox
to0pt{\kern0.35\wd0\vrule height0.45\ht0\hss}\raise0.05\ht0\hbox
to0pt{\kern0.5\wd0\vrule height0.45\ht0\hss}\box0}}
{\setbox0=\hbox{$\scriptscriptstyle\rm S$}\hbox{\raise0.5\ht0\hbox
to0pt{\kern0.4\wd0\vrule height0.45\ht0\hss}\raise0.05\ht0\hbox
to0pt{\kern0.55\wd0\vrule height0.45\ht0\hss}\box0}}}}
\def\bbbz{{\mathchoice {\hbox{$\sf\textstyle Z\kern-0.4em Z$}}
{\hbox{$\sf\textstyle Z\kern-0.4em Z$}}
{\hbox{$\sf\scriptstyle Z\kern-0.3em Z$}}
{\hbox{$\sf\scriptscriptstyle Z\kern-0.2em Z$}}}}

\def \bFq {\overline \F_q}

\newtheorem{theorem}{Theorem}
\newtheorem{lemma}[theorem]{Lemma}

\newtheorem{cor}[theorem]{Corollary}

\def\squareforqed{\hbox{\rlap{$\sqcap$}$\sqcup$}}
\def\qed{\ifmmode\squareforqed\else{\unskip\nobreak\hfil
\penalty50\hskip1em\null\nobreak\hfil\squareforqed
\parfillskip=0pt\finalhyphendemerits=0\endgraf}\fi}

\newcommand{\Card}[1]{\# #1}
\def\comp{\mathrm{comp}}

\newcommand{\ignore}[1]{}

\def\tGf{\Gamma_e^*} 

\def \F{{\bbbf}}

\def \Q{{\bbbq}}

\def\\{\cr}
\def\({\left(}
\def\){\right)}
\def\fl#1{\left\lfloor#1\right\rfloor}

\DeclareMathOperator{\cyc}{Cyc}
\DeclareMathOperator{\val}{val}

\begin{document}

\title{Functional Graphs of Polynomials over Finite Fields}

\author[Konyagin]{Sergei V.~Konyagin}
\address{Steklov Mathematical Institute,
8, Gubkin Street, Moscow, 119991, Russia} \email{konyagin@mi.ras.ru}

\author[Luca]{Florian~Luca}
\address{School of Mathematics, University of the Witwatersrand, P. O. Box Wits 2050, South Africa and
Mathematical Institute, UNAM Juriquilla, 76230 Santiago de Quer\'etaro, M\'exico}
\email{florian.luca@wits.ac.za}

\author[Mans]{Bernard Mans}
\address{Department of Computing, Macquarie University, Sydney, NSW 2109, Australia}
\email{bernard.mans@mq.edu.au}

\author[Mathieson]{Luke Mathieson}
\address{Department of Computing, Macquarie University, Sydney, NSW 2109, Australia}
\email{luke.mathieson@mq.edu.au}

\author[Sha]{Min Sha}
\address{School of Mathematics and Statistics, University of New South Wales, Sydney, NSW 2052, Australia}
\email{shamin2010@gmail.com}

\author[Shparlinski]{Igor E. Shparlinski}
\address{Department of Computing, Macquarie University, Sydney, NSW 2109, Australia}
\email{igor.shparlinski@unsw.edu.au}

\begin{abstract} Given a function $f$ in a finite field $\F_q$ 
of $q$ elements,  
we define
the functional graph of $f$ as a directed graph on $q$ nodes
labelled by the elements of $\F_q$ where
there is an edge from $u$ to $v$ if and only if $f(u) = v$.
We obtain some theoretic estimates on the number of non-isomorphic graphs
generated by all polynomials of a given degree. We then develop a simple and practical algorithm to
test the isomorphism of quadratic polynomials that has linear  memory and time complexities.
Furthermore, we extend this isomorphism testing algorithm to the general case of functional graphs,
and prove that, while its time complexity increases only slightly,
 its memory complexity remains linear.
 We exploit this algorithm to provide an upper bound on the number of
 functional graphs corresponding to
polynomials of degree $d$ over $\F_q$.
Finally, we present some numerical results and compare function graphs of quadratic
polynomials  with
 those generated  by random maps and pose interesting new problems.
\end{abstract}



\maketitle

\section{Introduction}

Let $\F_q$ be the finite field of $q$ elements
and 
of characteristic $p$.
For a function $f:\F_q \to \F_q$ we define
the functional graph of $f$ as a directed graph $\cG_f$ on $q$ nodes
labelled by the elements of $\F_q$ where
there is an edge from $u$ to $v$ if and only if $f(u) = v$.

Clearly each connected component of $\cG_f$ contains one cycle
(possibly of length 1 corresponding to a fixed point)
with several trees attached to some of the cycle nodes.

Here we are mostly interested in the graphs $\cG_f$ associated with
polynomials $f \in \F_q[X]$ of given degree $d$.

Some of our motivation comes from the natural desire to better understand
Pollard's $\rho$-algorithm (see~\cite[Section~5.2.1]{CrPom}).
We note that although this algorithm has been used and explored
for decades, there is essentially only one theoretic result
due to Bach~\cite{Bach}. In fact, even a heuristic model adequately
describing this algorithm is not quite clear, as the model of
random maps, analysed by Flajolet and Odlyzko~\cite{FlOdl}, does not take
into account the restrictions on the number of preimages. The model
analysed by MacFie and Panario~\cite{MFPan} approximates Pollard's
algorithm better but it perhaps still  does not capture it in full.
Polynomial maps can also be considered as building blocks for constructing
hash functions. For these applications,  it is important to
understand the intrinsic randomness of such maps.

Further motivation to investigation of  the graphs $\cG_f$ comes from the
theory of dynamical systems, as  $\cG_f$ fully encodes many of the dynamical
characteristics of the map $f$, such as the distribution of period (or cycle)
and pre-period lengths.

In particular, we denote by $N_d(q)$ the number of distinct (that is,
non-isomorphic) graphs $\cG_f$ generated by all polynomials
$f \in \F_q[X]$ of  degree $\deg f = d$. Since there are exactly $(q-1)q^d$ polynomials of degree $d$, we have a trivial upper bound:
$$
N_d(q) < q^{d+1}.
$$
Here, we use some ideas of Bach and Bridy~\cite{BaBr} together with some
new ingredients to obtain nontrivial bounds on $N_d(q)$.

We also design simple and practical, yet efficient, algorithms to test the isomorphism
of graphs $\cG_f$ and $\cG_g$ associated with two maps $f$ and $g$.
Furthermore, we design an efficient algorithm that generates a unique
label for each functional graph. We use these algorithms to design
an efficient procedure to list all $N_d(q)$ non-isomorphic graphs
generated by all the polynomials
$f \in \F_q[X]$ of  degree $\deg f = d$.

We conclude by presenting some numerical results for functional graphs of
quadratic polynomials. These results confirm that many (but not all, see below)
 basic characteristics of
these graphs, except the total number
of inner nodes, resemble those generated by random maps,
as analysed by~\cite{FlOdl}. A probabilistic model
of the distribution of cycles for functional graphs generated
by polynomials (and more generally, by rational functions) has been also developed
and numerically verified in~\cite{BGHKST}, see also~\cite{FlGar}.
Here, besides cycle lengths, we also
examine other characteristics of functional graphs generated by
quadratic polynomials, for example, the number of connected
components and the distribution of their sizes.
Furthermore, these numerical results exhibit some interesting statistical properties of
the graphs $\cG_f$ for which either there is no model in the setting
of random graphs, or they deviate, in a regular way, from such a model.

We also note that the periodic structure of functional graphs associated with monomial maps
$x\mapsto x^d$ over finite fields and rings
has been extensively studied (see~\cite{ChouShp,FPS,Gass,KurPom,MartPom,ShaHu,SoKr,VaSha,ZubTar}
and references therein). However, these graphs are expected to be  very
different from those associated with generic polynomials.

In characteristic zero, graphs generated by preperiodic points of a map $\psi$
(that is, by points that lead to finite orbits under iterations of $\psi$),
have also been studied, (see, for example,~\cite{DFK,Fab,FlPoSc,Mort,MoSi,Poon}).

We note that throughout the paper all implied constants in
$``O"$ symbols are {\it absolute\/}, unless stated otherwise.

\section{Bounds on the number of distinct functional graphs of polynomials}

\subsection{Upper bound}

To estimate $N_d(q)$ from the above, we use an idea of Bach and Bridy~\cite{BaBr}
which is based on the observation that for any polynomial
automorphism $\psi$ the composition map $\psi^{-1}\circ f \circ \psi$
has the same functional graph as $f$. So the idea is that if we can count the polynomials that are inequivalent under affine conjugations, this   gives an upper bound for the number of 
dynamically inequivalent polynomials 
and therefore also for $N_d(q)$. 

Thus, it needs to be shown that for any $d$
there exists a rather
small set of polynomials $\cF_d$ such that for any
polynomial $f \in \F_q[X]$ of  degree $d$ there is a
polynomial automorphism
$\psi$ such that  $\psi\circ f \circ \psi^{-1} \in \cF_d$. Then
we have $N_d(q) \le \# \cF_d$.
To construct the set $\cF_d$ we introduce a group of certain
transformations on the set of polynomials and show that
\begin{itemize}
\item polynomials in each orbit generate isomorphic
graphs;
\item each orbit is sufficiently long;
\item most of the orbits are of the size of the above group.
\end{itemize}
This approach has been used in~\cite{BaBr} for $d=2$ and
even $q = 2^n$, in which case it is especially effective and
leads to the bound
\begin{equation}
\label{eq:BB bound}
N_2(2^n) = \exp\(O\(\frac{n}{\log \log n}\)\) = q^{O\(1/\log \log \log q\)}.
\end{equation}

For general pairs $(d,q)$ this approach loses some of its power
but still leads to nontrivial results, explicitly in both $d$ and $q$.
Recall that $p$ is the characteristic of $\F_q$.

\begin{theorem}
\label{thm:Nd bound U}
For any $d\ge 2$ and $q$, we have
\begin{equation*}
\begin{split}
N_d(q)
&\le \left\{\begin{array}{ll}
q^{d-1}+(s-1)q^{d-1-\varphi(d-1)},& \text{if $p\nmid d$,}\\
q^{d-1}+(s-1)q^{d-1-\varphi(d-1)}+(q-1)q^{d/p-1},&\text{if $p \mid d$},
\end{array}
\right.
\end{split}
\end{equation*}
where $s=\gcd(q-1,d-1)$, and $\varphi$ is Euler's totient function.
In particular, $N_2(q)\le 2q-1$, and for any $d\ge 3$ we have $N_d(q)\le 3q^{d-1}$. Furthermore, $N_d(q)\le q^{d-1}$ if $p\nmid d$ and $\gcd(q-1,d-1)=1$.
\end{theorem}

\begin{proof}
For $\lambda \in \F_q^*$ and $\mu \in \F_q$, we define the  automorphism
\begin{equation}
\label{eq:phi}
\phi_{\lambda,\mu}:\  X \mapsto \lambda X + \mu
\end{equation}
with inverse $\phi_{\lambda,\mu}^{-1}: \  X \mapsto \lambda^{-1}(X - \mu)$. Particularly, these  automorphisms form a group of order $(q-1)q$ in the usual way, which acts on the set of polynomials of degree $d$ as the map
$$
f(X) \to \phi_{\lambda,\mu}^{-1} \circ f \circ \phi_{\lambda,\mu}(X).
$$
 The number of the orbits of this group action can be calculated by the Burnside counting formula. This implies that
\begin{equation} \label{Burnside}
N_d(q)\le \frac{1}{(q-1)q}\sum_{(\lambda,\mu)} M_d(\lambda,\mu),
\end{equation}
where the sum runs through all the pairs $(\lambda,\mu)\in \F_q^*\times \F_q$, and $M_d(\lambda,\mu)$ is the number of polynomials of degree $d$  fixed by $\phi_{\lambda,\mu}$.

Trivially, we have
\begin{equation}
\label{eq:mu=0}
M_d(1,0)=(q-1)q^d.
\end{equation}

In the following, we want to estimate $M_d(\lambda,\mu)$ by fixing a pair $(\lambda, \mu)\in \F_q^*\times \F_q\setminus\{(1,0)\}$.

For any polynomial $f$ of degree $d$ satisfying
 $\phi_{\lambda,\mu}^{-1} \circ f \circ \phi_{\lambda,\mu}(X)= f(X)$,
 we have
$$
f(\lambda X + \mu) = \lambda f(X) + \mu.
$$
Comparing the leading coefficients we
derive
\begin{equation}
\label{eq:bad lambda}
\lambda^{d-1} = 1,
\end{equation}
which implies that
\begin{equation}
\label{eq: good lambda}
M_d(\lambda,\mu)=0
\end{equation}
for those pairs $(\lambda,\mu)$ not satisfying~\eqref{eq:bad lambda}.

First, suppose that $\lambda=1$. Note that $\mu \ne 0$.
Comparing the coefficients of $X^{d-1}$ in  $f(X + \mu)$
and  $f(X) + \mu$,
we obtain
\begin{equation}
\label{eq:da_d}
da_d = 0.
\end{equation}
Thus, $p\mid d$.
Moreover, comparing the coefficients of $X^{j-1}$ in  $f(X + \mu)$
and  $f(X) + \mu$  for every $j=1, \ldots, d$,
we also obtain
relations of the form
$$
ja_j\mu= F_j(a_d, \ldots, a_{j+1}, \mu), \qquad j=1, \ldots, d,
$$
for some polynomials
$$
F_j\in \F_q[Z_d, \ldots, Z_{j+1}, V],
$$
where in the case $j=d$ we have $F_d(V) = 0$, which corresponds to~\eqref{eq:da_d}.
In particular, for every $j=1, \ldots, d$ with $\gcd(j,p)=1$,
we see that $a_j$ is uniquely defined by $a_d, \ldots, a_{j+1}, \mu$.
Hence, for $\mu \ne 0$ we get that
\begin{equation}
\begin{split}
\label{eq:lambda =1}
M_d(1,\mu)
&\le \left\{\begin{array}{ll}
0,& \text{if $p\nmid d$,}\\
(q-1)q^{d/p},&\text{if $p \mid d$}.
\end{array}
\right.
\end{split}
\end{equation}

Assume now that $\lambda^{d-1} = 1$ but $\lambda \ne 1$, which implies that $d\ge 3$.
We see that for every $j = 0,1, \ldots, d$ there are polynomials
$$
G_j\in \F_q[Z_d, \ldots, Z_{j+1}, U, V]
$$
such that
$$
a_j(\lambda^j - \lambda) = G_j(a_d, \ldots, a_{j+1}, \lambda, \mu).
$$
Since $\lambda \ne 0,1$, and $\lambda^{d-1} = 1$, it follows
that for every $j$ with $\gcd(j-1,d-1)=1$ we have $\lambda^{j} \ne \lambda$
and thus $a_j$ is uniquely defined by $a_d, \ldots, a_{j+1}, \lambda, \mu$. So, for $d\ge 3$ and any pair $(\lambda,\mu)$ satisfying $\lambda^{d-1} = 1$ and $\lambda \ne 1$, we have
\begin{equation}
\label{eq:lambda not 1}
M_d(\lambda,\mu)\le (q-1)q^{d-1-\varphi(d-1)}.
\end{equation}
Notice that since  $\lambda^{d-1} = 1$ and $\lambda \ne 1$,  the element $\lambda$ can take at most $\gcd(q-1,d-1)-1$ values.

Using~\eqref{Burnside} together with~\eqref{eq:mu=0}, \eqref{eq: good lambda}, \eqref{eq:lambda =1}
and~\eqref{eq:lambda not 1}, we complete the proof.
 \end{proof}

\subsection{Lower bound: Idea of the proof}

Here we give a lower bound on $N_d(q)$ in the case of $\gcd(d,q-1) \ge 2$ and $\gcd(d-1,q)=1$.
In particular, this bound shows that~\eqref{eq:BB bound} does not hold for fields of odd characteristic.

The idea is based on the following observation.
Let $\cH_a$ be the functional graph of $f_a(X) = X^d + a \in \F_q[X]$ with $a \in \F_q^*$.
We note that the node $a$ is the only node with in-degree $1$, because the in-degree of every other node is
$$
e = \gcd(d,q-1) \ge 2.
$$

We now define the iterations of $f_a$
$$
f_a^{(0)}(X) = X \mand  f_a^{(k)}(X) = f_a\(f_a^{(k-1)}(X)\), \quad k=1, 2, \ldots ,
$$
and
consider the path of length $J$
\begin{equation}
\label{eq:path}
a =  f_a^{(0)}(a)\to f_a(a)=f_a^{(1)}(a) \to \cdots\to f_a^{(J)}(a)
\end{equation}
originating from $a$.
Then, the $(j+1)$-th node of this path has $e-1$  edges
towards it from  $\gamma f_a^{(j)}(a)$, where $\gamma$ runs through the elements
of the set
$$
\tGf = \Gamma_e \setminus \{1\},
$$
where
$$
\Gamma_e = \{\gamma\in \F_q~:~\gamma^e = 1\}.
$$
Finally, we observe that  $\gamma f_a^{(j)}(a)$
is an inner node if  and only if  the equation
$$
z^d + a = \gamma f_a^{(j)}(a)
$$
has a solution.

We now note that if two graphs $\cH_a$ and $\cH_b$ are isomorphic then, since $a$ and $b$
are unique nodes with the in-degree $1$ in $\cH_a$ and $\cH_b$, respectively,
the paths of the form~\eqref{eq:path} originating at $a$ and $b$, and their  neighbourhoods have to be isomorphic too.

For $j =1, 2, \ldots$, we define $\eta_j(a)$ as the number of $\gamma \in \tGf$
for which $\gamma f_a^{(j)}(a)-a$ is an $e$-th power
nonresidue.
Thus,  $\eta_j(a)$ is the number of leaves amongst the nodes
$\gamma f_a^{(j)}(a)$,
$\gamma \in \tGf$.

Therefore, for any $J$,  the number of distinct vectors
\begin{equation}
\label{eq:vect}
\(\eta_1(a),  \ldots, \eta_J(a) \), \qquad a \in \F_q^*,
\end{equation}
gives a lower bound on $N_d(q)$. Our approach is to find a proper choice of
$J$ when $q$ is sufficiently large such that each $\eta_j(a)$ ($j\ge 2$) can run through the set $\{0,1,\ldots,e-1\}$, then we can get a lower bound of the form
$$
N_d(q)\ge e^{J-1}.
$$
The idea is illustrated in Figure~\ref{fig:idea}, where each ``\,$i$\,'' ($1\le i \le e-1$) in the circles represents a node defined by some $\gamma f_a^{(j)}(a)$, $\gamma\in \tGf$.
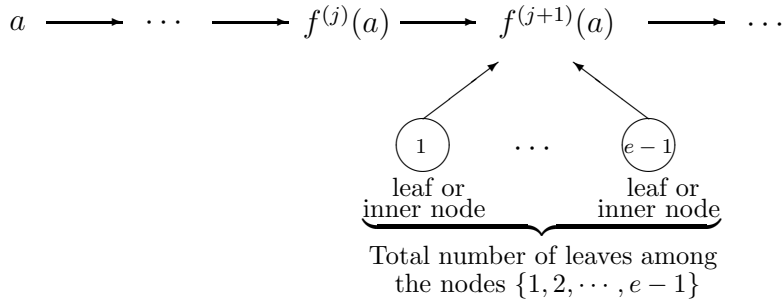
\begin{figure}[h]
\begin{center}
\setlength{\unitlength}{1cm}
\begin{picture}(20,4)(1,-1)

\put(2.0,2.53){$a$}
\put(2.5,2.63){\vector(2,0){1}}
\put(3.8,2.53){$\cdots$}
\put(4.7,2.63){\vector(2,0){1}}
\put(5.9,2.5){$f^{(j)}(a)$}  
\put(7.2,2.63){\vector(2,0){1}}
\put(8.5,2.5){$f^{(j+1)}(a)$}
\put(10.5,2.63){\vector(2,0){1}}
\put(11.8,2.5) {$\cdots$}

\put(7.5,1){\circle{0.75}}
\put(7.4,0.9){\tiny{1}}
\put(7.5,1.35){\vector(4,3){1}}
\put(8.7,0.85){$\cdots$}
\put(10.5,1){\circle{0.75}}
\put(10.17,0.9){\tiny{$e-1$}}
\put(10.5,1.35){\vector(-4,3){1}}
\put(6.7,0.2){$\underbrace{\large{\substack{ \textrm{ leaf or} \\ \textrm{inner node}}} \qquad\quad \substack{\textrm{ leaf or} \\ \textrm{inner node}}}$}
\put(6.78,-0.8){${\large\substack{\textrm{Total number of leaves among}\\ \textrm{ the nodes $\{1,2,\cdots, e-1\}$}}}$}
\end{picture}
\end{center}
\caption{Idea of lower bound} \label{fig:idea}
\end{figure}

We can express the appearance
of a particular ``pattern'' among the leaves~\eqref{eq:vect} algebraically, and use
the Weil bound of multiplicative character sums (see~\cite[Theorem~11.23]{IwKow})
to show that, when $J$ is not too large, all possible patterns appear;
see Theorem~\ref{thm:Nd bound L} and its proof for more details.
Note that this
is similar to the well-known use of the Weil bound for showing that a sequence $\{1, \ldots, p\}$ contains
any desired pattern of $J$ consecutive residues and non-residues.

\subsection{Lower bound: Technical details}

In order to realise the above approach,
we need several technical statements.

Let us consider the sequences of polynomials
$$
F_0(X) = X \mand  F_k(X) = \(F_{k-1}(X)\)^d + X, \quad k=1, 2, \ldots,
$$
and
also
$$
G_{k,\gamma}(X)  = \gamma F_k(X) - X.
$$
Note that the roots of $G_{k,\gamma}$ are exactly those $z \in  \bFq$ for which
$F_k$ twists $z$ by $\gamma^{-1}$.

We now investigate some arithmetic properties of polynomials $G_{k,\gamma}$
which we present in larger generality than we actually need for our
purposes.

\begin{lemma}
\label{lem:cong}
For any positive integers $k$ and $h$ and $\gamma, \delta \in \Gamma_e$, we have
$$
G_{k+h,\gamma} \equiv G_{h,\gamma} \pmod {G_{k,\delta}}.
$$
\end{lemma}

\begin{proof} We fix $\gamma, \delta \in \Gamma_e$ and prove the desired
statement by induction on $h=1, 2, \ldots$.

We note that for $ \delta \in \Gamma_e$  we have
\begin{equation}
\label{eq:Fd G}
\( F_{k}(X) \)^d = \(\delta^{-1}\(G_{k,\delta}(X)  + X\)\)^d = \(G_{k,\delta}(X)  + X\)^d.
\end{equation}
For $h=1$ we have $G_{1,\gamma}(X) = \gamma X^d +(\gamma-1) X$.
Hence, using~\eqref{eq:Fd G}, we derive
\begin{equation*}
\begin{split}
G_{k+1,\gamma}(X) &= \gamma  \(\( F_{k}(X) \)^d +X\) - X \\
 & = \gamma\(G_{k,\delta}(X)  + X\)^d + (\gamma-1) X \\
& \equiv \gamma X^d +(\gamma-1) X \equiv G_{1,\gamma}(X)  \pmod {G_{k,\delta}(X)},
\end{split}
\end{equation*}
so the desired congruence holds for $h=1$.

Now assume that it also holds for $h=\ell$.  Then
\begin{equation*}
\begin{split}
G_{k+\ell+1,\gamma}(X) & = \gamma \(G_{k+\ell,\gamma}(X)  +X\)^d + (\gamma-1) X \\
& \equiv
\gamma \(G_{\ell,\gamma}(X) + X\)^d + (\gamma-1)  X \\
&\equiv G_{\ell+1,\gamma}(X) \pmod {G_{k,\delta}(X)},
\end{split}
\end{equation*}
which implies the desired result.
\end{proof}

\begin{lemma}
\label{lem:gcd}
For any positive integers $k$ and $m$, we have
$$
\gcd(G_{k,\gamma},G_{m,\gamma}) = G_{\gcd(k,m),\gamma}.
$$
\end{lemma}

\begin{proof}
If $k = m$, then there is nothing to prove.
Otherwise we note that for $m > k$, Lemma~\ref{lem:cong}
implies $G_{m,\gamma} \equiv G_{m-k,\gamma} \pmod {G_{k,\gamma}}$. Thus
$$
\gcd(G_{k,\gamma},G_{m,\gamma}) = \gcd(G_{k,\gamma},G_{m-k,\gamma}),
$$
which immediately implies the desired result.
\end{proof}

Now, from  Lemma~\ref{lem:gcd} we derive
that for $d=2$ (and odd $q$), products of polynomials $ G_{j,-1}$ over distinct integers are not perfect squares; see Lemma~\ref{lem:sqr}.

As usual, we use $\bFq$ to denote the algebraic closure of $\F_q$.

\begin{lemma}
\label{lem:sqr}
For $d=e=2$, odd $q$, and any non-empty set $\cJ$ of positive integers, we have
$$
\prod_{j \in \cJ} G_{j,-1}   \ne P^2
$$
for any polynomial $P \in \bFq[X]$.
\end{lemma}

\begin{proof}

Assume that $m$ is the largest element of $\cJ$. 
Since $d=e=2$, we have 
$$
F_0(X)=X,\quad F_k(X)=F_{k-1}(X)^2+X,\quad G_{k,-1}(X)=-F_k(X)-X, 
$$
and thus 
$$
G_{1,-1}(X)=-X(X+2) \mand G_{2,-1}=-X(X+2)(X^2+1).
$$
So, the cases $m=1$ and $m=2$ can be verified by direct
calculations. 

Now, we assume that $m \ge 3$.
 It suffices to show that $G_{m,-1}(X)$ has a simple root which is not a root of
$$
Q_{m-1}(X)=\prod_{j=1}^{m-1} G_{j,-1}(X).
$$
Let $f(m)$ be the number of distinct roots of $\gcd(G_{m,-1}(X), Q_{m-1}(X))$.
By Lemma~\ref{lem:gcd}, these distinct roots 
are to be found among the distinct roots of
$$
\prod_{1\le k\le m-1} \gcd(G_{m,-1}(X),G_{k,-1}(X))=\prod_{1\le k \le m-1} G_{\gcd(m,k),-1}(X),
$$
and the distinct roots of this last polynomial are the same as the distinct roots of the polynomial
$$
\prod_{\substack{k\mid m\\ 1\le k<m}} G_{k,-1}(X),
$$
which implies that 
\begin{equation}
\label{eq:fm}
f(m) \le \sum_{\substack{k\mid m\\ 1\le k<m}} 2^k \le  \sum_{1\le k \le m/s} 2^k = 2^{m/s+1}-2,
\end{equation}
where $s$ is the minimal prime factor of $m$. More precisely, since 
the polynomial $G_{1,-1}(X)$ divides any other polynomial $G_{k,-1}(X)$,  $k\ge 1$, we have 
\begin{equation}
\label{eq:f6}
f(6)\le \deg G_{1,-1}+\deg G_{2,-1}+\deg G_{3,-1}-2\deg G_{1,-1}=10.
\end{equation}

 Now let us write 
$$
G_{m,-1}(X)=A(X)^2 B(X),
$$
where $A(X),B(X)\in \F_q[X]$ are monic polynomials and  $B(X)$ has only simple roots. 

We claim that 
\begin{equation}
\label{eq:deg B f >}
\deg B(X) > f(m) 
\end{equation}
when $m\ge 5$. 
So, if $m\ge 5$,  then $G_{m,-1}(X)$ has a root of odd multiplicity which 
is not a root of $Q_{m-1}(X)$, thus the desired result follows. 

We prove the claim by contradiction. Hence, we suppose that 
\begin{equation}
\label{eq:deg B f le}
\deg B(X)\le f(m) \mand m \ge 5.
\end{equation} 

Note that since $X=0$  is a simple root of $G_{m,-1}(X)$, we have $\deg B(X)\ge 1$. Since $A(X)$ divides $\gcd(G_{m,-1}(X),G_{m,-1}^{\prime}(X))$ and 
$G_{m,-1}(X)=-F_{m-1}(X)^2-2X$, we obtain 
$$
F_{m-1}(X)^2+2X \equiv F_{m-1}(X)F_{m-1}^{\prime}(X)+1 \equiv 0 \pmod{A(X)},
$$
which yields that 
\begin{align*}
0&\equiv F_{m-1}(X)^2F_{m-1}^{\prime}(X)+F_{m-1}(X) \\
&\equiv -2XF_{m-1}^{\prime}(X)+F_{m-1}(X)  \pmod{A(X)}.
\end{align*}
Moreover, since 
\begin{equation}
\label{eq:congX3}
F_{m-1}(X)=F_{m-2}^2(X)+X\equiv  X^2+X\pmod {X^3}
\end{equation}
and 
\begin{align*}
F_{m-1}'(X)& =2F_{m-2}(X) F_{m-2}'(X)+1\\
& \equiv 2 X (2X+1)+1\equiv 2X+1\pmod {X^2},
\end{align*}
it follows that 
\begin{align*}
F_{m-1}(X)-2XF_{m-1}'(X) & \equiv  X^2+X-2X(2X+1) \\
& \equiv  -3X^2-X\pmod {X^3},
\end{align*}
so this last polynomial is not the zero polynomial. 
In particular, there is a non-zero polynomial $C(X)\in \F_q[X] $ such that 
$$
F_{m-1}(X)-2XF_{m-1}^{\prime}(X)=A(X)C(X).
$$
Because the degree of $ F_{m-1}(X)-2XF_{m-1}^{\prime}(X)$ is at most $2^{m-1}$, we deduce that 
\begin{equation}
\label{eq:deg C}
\deg C(X) \le 2^{m-1}-\deg A(X) = \frac{1}{2}\deg B(X). 
\end{equation}

We can also write 
$$
A(X)= \frac{F_{m-1}(X)-2XF_{m-1}^{\prime}(X)}{C(X)}.
$$
Thus, we get that
 $$
 -F_{m-1}(X)^2-2X=G_{m,-1}(X)=\frac{(F_{m-1}(X)-2XF_{m-1}'(X))^2  B(X)}{C(X)^2},
 $$
 and  
 \begin{equation}
 \label{eq:3}
 -(F_{m-1}(X)^2+2X)C(X)^2=(F_{m-1}(X)-2XF_{m-1}'(X))^2  B(X).
 \end{equation}
 Using the relation
 $$
 F_{m-1}(X)=F_{m-2}(X)^2+X,\qquad F_{m-1}'(X)=2F_{m-2}(X)F_{m-2}'(X)+1,
 $$
 we reduce~\eqref{eq:3} modulo $F_{m-2}(X)$ to obtain  
 $$
 -(X^2+2X)C(X)^2\equiv X^2 B(X)\pmod {F_{m-2}(X)}.
 $$
Notice that by~\eqref{eq:deg C} the polynomial on the left hand side has degree at most $\deg B(X)+2$, and the polynomial on the right hand side has degree $\deg B(X)+2$.  
 Thus, in view of~\eqref{eq:deg B f le}, if $\deg F_{m-2}(X)>f(m)+2$, then the above congruence must in fact be an equality. Using~\eqref{eq:fm} and~\eqref{eq:f6}, the above inequality is satisfied if 
 $$
2^{m-2}> 2^{m/s+1} \quad \textrm{or} \quad m=6,
 $$
 where $s$ is the minimal prime factor of $m$. 
The above inequality is also true for $m=5$ and any integer $m\ge 7$. 

 So, if $m\ge 5$, we must have 
  $$
  -(X^2+2X) C(X)^2=X^2  B(X).
  $$
Furthermore,
  $$
  \frac{B(X)}{C(X)^2}=-\frac{X+2}{X},
  $$
  so that 
 \begin{align*}
  -F_{m-1}(X)^2-2X & =  \frac{(F_{m-1}(X)-2XF_{m-1}'(X))^2 B(X)}{C(X)^2}\\
  & =  -\frac{(X+2)}{X} (F_{m-1}(X)-2XF_{m-1}'(X))^2,
  \end{align*}
 and then 
 $$
 X(F_{m-1}(X)^2+2X)=(X+2)(F_{m-1}(X)-2XF_{m-1}'(X))^2.
 $$
 We reduce the above relation modulo $F_{m-3}(X)$ using the fact that
 \begin{align*}
 F_{m-1}(X)=F_{m-2}(X)^2+X &=(F_{m-3}(X)^2+X)^2+X \\
 &\equiv X^2+X \pmod {F_{m-3}(X)}
 \end{align*}
 and 
 \begin{align*}
 F_{m-1}'(X) &=2(F_{m-3}^2(X)+X)(2F_{m-3}(X)F_{m-3}'(X)+1)+1 \\
 &\equiv 2X+1\pmod {F_{m-3}(X)},
 \end{align*}
 to get that
 $$
 X((X^2+X)^2+2X)\equiv (X+2)(X^2+X-2X(2X+1))^2 \pmod{F_{m-3}(X)}.
 $$
 This leads to 
\begin{equation}
\label{eq:cong}
 8X^5+22X^4+12X^3\equiv 0\pmod {F_{m-3}(X)},
\end{equation}
 which is impossible when $m\ge 5$. Indeed, for $k \ge 0$ we obviously have 
$F_k(X) \equiv X \pmod {X^2}$, so~\eqref{eq:cong} implies 
$$
 8X^3+22X^2+12X\equiv 0\pmod {F_{m-3}(X)}
$$
(see also~\eqref{eq:congX3}), which is impossible as $m\ge 5$  we have
 $$
 \deg(8X^3+22X^2+12X) = 3 < 2^{m-3}=\deg F_{m-3}(X),
 $$
so~\eqref{eq:deg B f le} is impossible. Thus, \eqref{eq:deg B f >} 
holds and we have the desired result for such values of $m$.
 
Finally, in order to finish the proof, it remains to handle the cases $m=3,4$. We need to show 
that for $m=3,4$, the polynomial $G_{m,-1}(X)$ 
 has a simple root which is not a root of $G_{k,-1}(X)$ for any proper divisor $k$ of $m$. For this we treat such $G_{m,-1}(X)$ as polynomials with integer coefficients and compute their discriminants. Notice that only the prime factors $\ell$ of such discriminants can be characteristics of fields where $G_{m,-1}(X)$ has double roots. Then, we factor such 
 $G_{m,-1}(X)$ over the corresponding finite fields $\F_\ell$ for such primes $\ell$, and can see that 
 there is always  an irreducible factor of $G_{m,-1}(X)$ with multiplicity $1$ which is not a factor of any
 $G_{k,-1}(X)$ ($k|m, k<m$) over $\F_\ell$. This completes the proof. 
\end{proof}

In the case when $d > 2$ we can study the arithmetic structure
of the polynomials  $G_{k,\gamma}$ by using Mason's theorem~\cite[page~156, Corollary]{Mason}, which
gives a polynomial version of the $ABC$-conjecture (also called the Mason-Stothers theorem
since it has  also been discovered independently by Stothers~\cite[Theorem~1.1]{Stothers}, see also~\cite{Sny}).
We present it in Lemma~\ref{lem:ABC} below.

For a  polynomial $F \in \F_q[X]$ we use $\rad(F)$ to denote the
product of all monic irreducible divisors of $F$.

\begin{lemma}
\label{lem:ABC}
Let $A$, $B$, $C$ be nonzero polynomials in $\F_q[X]$
satisfying $A+B+C = 0$ and $\gcd\(A, B, C\)= 1$. If
$\deg A \ge \deg \rad (ABC)$, then $A'=B'=C'= 0$.
\end{lemma}

We are now ready to prove our main technical statement that
we use for $d \ge 3$ which asserts that some
general products of polynomials $ G_{k,\gamma}$ over distinct  integers
are not perfect $e$-th powers.

\begin{lemma}
\label{lem:e pow}
Suppose that $\gcd(d-1,q)=1$. Then, for $d\ge 3$, $e=\gcd(d,q-1) \ge 2$, any $J\ge 2$ and  any collection of integers not all equal to zero
$$
\cA = \{\alpha_{j,\gamma} \in \{0, \ldots, e-1\}~:~2\le j \le J, \ \gamma \in \tGf\},
$$
 we have
$$
 \prod_{j=2}^J\, \prod_{\gamma \in \tGf} G_{j,\gamma}^{\alpha_{j,\gamma}}  \ne P^e
$$
for any polynomial $P \in \bFq[X]$.
\end{lemma}

\begin{proof}
Clearly,
we observe that for $j=1,2, \ldots$ we have $X \mid G_{j,\gamma}$ but $X^2 \nmid G_{j,\gamma}$ for any $\gamma \in \tGf$.
Hence, define
$$
G_{j,\gamma}^*= G_{j,\gamma}/X, \qquad j=1,2, \ldots, \   \gamma  \in \tGf.
$$
By counting the common roots, for  distinct  $\gamma, \delta \in \tGf$ we have
\begin{equation}
\label{eq:gcd X}
\gcd(G_{j,\gamma}^*, G_{j,\delta}^*)=1.
\end{equation}
Therefore, applying Lemma~\ref{lem:cong}, for any positive integers $k$
and $h$ and $\delta \in \tGf$ we obtain
\begin{equation}
\label{eq:gcd kh k}
\begin{split}
\deg \gcd\(G_{k+h,\delta}^*,\prod_{\gamma \in \tGf} G_{k,\gamma} \)& =
\deg \gcd\(G_{k+h,\delta}^*,\prod_{\gamma \in \tGf} G_{k,\gamma}^*  \) \\
 = \deg &\gcd\(G_{h,\delta}^*,\prod_{\gamma \in \tGf} G_{k,\gamma}^* \)
\le d^h -1 .
\end{split}
\end{equation}
For any $k\ge 1$ and $\gamma \in \tGf$, since $X\mid F_{k-1}$ and  $\gcd(d-1,q)=1$, we get that $(F_{k-1}^d/X)^{'}\ne 0$; then applying Lemma~\ref{lem:ABC} with
$A=-G_{k,\gamma}^*$, $B =\gamma  F_{k-1}^d/X$ and $C= \gamma-1$,
we derive
\begin{equation*}
\begin{split}
d^k-1 < \deg\rad\(G_{k,\gamma}^* F_{k-1}^d/X\)
&= \deg\rad\(G_{k,\gamma}^* F_{k-1}\)\\
& \le \deg\rad\(G_{k,\gamma}^*\) + d^{k-1}.
\end{split}
\end{equation*}
Thus,
\begin{equation}
\label{eq: deg rad}
 \deg \rad\(G_{k,\gamma}^*\) \ge (d-1)  d^{k-1}, \quad k =1,2, \ldots, \ \gamma
\in \tGf.
\end{equation}
Denote
$$
Q_{J,\cA} = \prod_{j=2}^J\, \prod_{\gamma \in \tGf} G_{j,\gamma}^{\alpha_{j,\gamma}},
$$
and assume that $Q_{J,\cA}= P^e$ for some $P \in \bFq[X]$. Let
$$
\widetilde Q_{J,\cA} = \prod_{j=2}^J\, \prod_{\gamma \in \tGf}
G_{j,\gamma}^{\widetilde \alpha_{j,\gamma}},
$$
where $\widetilde\alpha_{j,\gamma}=e-\alpha_{j,\gamma}$ if
$\alpha_{j,\gamma}\neq0$ and $\widetilde\alpha_{j,\gamma}=0$ otherwise.
Then, since each $\alpha_{j,\gamma}\le e-1$, we have
$$P\mid \prod_{j=2}^J\, \prod_{\gamma \in \tGf,\,\alpha_{j,\gamma}\neq0}
G_{j,\gamma}.$$
Noticing that
$$\widetilde Q_{J,\cA}=\left(\frac
{\prod_{j=2}^J\prod_{\gamma \in \tGf,\,\alpha_{j,\gamma}\neq0}G_{j,\gamma}}P
\right)^e,$$
we conclude that $\widetilde Q_{J,\cA}$ is also a perfect $e$-th power.
Let $k\ge 2$ be the largest $j \in \{2, \ldots, J\}$ for which one
of the integers $\alpha_{j,\gamma}$, $\gamma \in \tGf$ is
positive.
Considering, if necessary, $\widetilde Q_{J,\cA}$ we can always assume that
$$
\alpha = \min_{\gamma \in \tGf}  \{\alpha_{k,\gamma}~:~\alpha_{k,\gamma}>0\} \le e/2.
$$
We now fix some $\delta \in  \tGf$ with $\alpha_{k,\delta}  = \alpha$,
$1\le\alpha\le e/2$.
If a polynomial $H \in \F_q[X]$ is such that the product
$H (G_{k,\delta}^*)^{\alpha_{k,\delta}}$ is a perfect $e$-th power, then
$$
 \rad(G_{k,\delta}^*)^e \mid H (G_{k,\delta}^*)^{\alpha_{k,\delta}}.
$$
So,
$$
\rad(G_{k,\delta}^*)^{e-\alpha_{k,\delta}}
\mid H \(G_{k,\delta}^*/\rad(G_{k,\delta}^*)\)^{\alpha_{k,\delta}},
$$
which, combining with~\eqref{eq: deg rad}, implies that
\begin{equation}
\label{eq: deg H}
\begin{split}
 \deg & \gcd\((G_{k,\delta}^*)^{e-\alpha_{k,\delta}}, H\)\\
& \ge\deg\gcd\(\rad(G_{k,\delta}^*)^{e-\alpha_{k,\delta}}, H\)\\
& \ge(e-\alpha_{k,\delta})\deg\rad(G_{k,\delta}^*)
-\alpha_{k,\delta}\(\deg G_{k,\delta}^*-\deg\rad(G_{k,\delta}^*)\)\\
& \ge e(d-1)d^{k-1}-\alpha_{k,\delta}(d^k-1) \ge \frac{e}{2} (d-2) d^{k-1}.
\end{split}
\end{equation}

If $k=2$, that is $J=2$ and 
$$
Q_{J,\cA}=\prod_{\gamma \in \tGf}G_{k,\gamma}^{\alpha_{k,\gamma}}=P^e,
$$
then by~\eqref{eq: deg H}, we obtain
$$
\deg \gcd\((G_{k,\delta}^*)^{e-\alpha_{k,\delta}}, Q_{J,\cA}/\(G_{k,\delta}^*\)^{\alpha_{k,\delta}}\) \ge \frac{e}{2} (d-2) d^{k-1}>0.
$$
However, combining~\eqref{eq:gcd X} with $x\nmid G_{k,\delta}^*$, we have
$$
\deg \gcd\((G_{k,\delta}^*)^{e-\alpha_{k,\delta}}, Q_{J,\cA}/\(G_{k,\delta}^*\)^{\alpha_{k,\delta}}\) =0,
$$
which leads to a contradiction. So, we must have $k\ne 2$, that is $k\ge 3$.

Now, combining~\eqref{eq:gcd X} with~\eqref{eq: deg H}, we have
\begin{equation}
\label{eq:gcd Large}
\begin{split}
 \deg & \gcd\(\(G_{k,\delta}^*\)^{e-\alpha_{k,\delta}},
\prod_{j=2}^{k-1}\, \prod_{\gamma \in \tGf}G_{j,\gamma}^{\alpha_{j,\gamma}}\)\\
& = \deg \gcd\(\(G_{k,\delta}^*\)^{e-\alpha_{k,\delta}},
Q_{J,\cA}/\(G_{k,\delta}^*\)^{\alpha_{k,\delta}}\)\\
& \ge \frac{e}{2} (d-2) d^{k-1}.
\end{split}
\end{equation}
On the other hand, we deduce that
\begin{equation*}
\begin{split}
\deg & \gcd \(\(G_{k,\delta}^*\)^{e-\alpha_{k,\delta}},  \prod_{j=2}^{k-1}\, \prod_{\gamma \in \tGf}
G_{j,\gamma}^{\alpha_{j,\gamma}}\) \\
& \le
\deg \gcd\(\(G_{k,\delta}^*\)^{e-\alpha_{k,\delta}},  \prod_{\gamma \in \tGf}
G_{k-1,\gamma}^{\alpha_{k-1,\gamma}}\) + \sum_{j=2}^{k-2}\sum_{\gamma \in \tGf}
\alpha_{j,\gamma} \deg G_{j,\gamma}\\
& \le
 \deg \gcd\(\(G_{k,\delta}^*\)^{e-1},  \prod_{\gamma \in \tGf}
G_{k-1,\gamma}^{e-1}\) + (e-1)^2 \sum_{j=2}^{k-2}  d^j.
\end{split}
\end{equation*}
Using~\eqref{eq:gcd kh k} (with $h=1$) we get
$$
\deg \gcd \(\(G_{k,\delta}^*\)^{e-1},  \prod_{\gamma \in \tGf}
G_{k-1,\gamma}^{e-1}\) \le(e-1)(d-1).
$$
Collecting the above estimates,
we obtain
\begin{equation}
\label{eq:gcd Small}
\begin{split}
 \deg & \gcd \(\(G_{k,\delta}^*\)^{e-\alpha_{k,\delta}},
\prod_{j=2}^{k-1}\prod_{\gamma \in \tGf}
G_{j,\gamma}^{\alpha_{j,\gamma}}\) \\
& \qquad \le (e-1)(d-1) + (e-1)^2 \frac{d^{k-1}-1}{d-1}.
\end{split}
\end{equation}
It is now easy to verify  that~\eqref{eq:gcd Small}
contradicts~\eqref{eq:gcd Large} when $d>3$. Indeed, combining~\eqref{eq:gcd Large} with~\eqref{eq:gcd Small}, we get
$$
\frac{e}{2} (d-2) d^{k-1}\le (e-1)(d-1) + (e-1)^2 \frac{d^{k-1}-1}{d-1}.
$$
Then, since $d>3$, we can get
$$
e(d-1) d^{k-1}\le (e-1)(d-1)^2 + (e-1)^2d^{k-1}-(e-1)^2,
$$
and thus
$$
(e-1)d^{k-1} \le (e-1)(d-1)^2-(e-1)^2,
$$
which is impossible by noticing that $k\ge 3$.

Finally, we take
$d=3$. Then, $e=3$. From~\eqref{eq: deg H}, we have
\begin{equation*}
\begin{split}
  \deg & \gcd \(\(G_{k,\delta}^*\)^{e-\alpha_{k,\delta}},
\prod_{j=2}^{k-1}\, \prod_{\gamma \in \tGf}
G_{j,\gamma}^{\alpha_{j,\gamma}}\) \\
& \ge e(d-1)d^{k-1}-\alpha_{k,\delta}(d^k-1) \ge 3^k+1,
\end{split}
\end{equation*}
which contradicts~\eqref{eq:gcd Small}.
The desired result now follows.
\end{proof}

Let $\cX_e$ be the group of all multiplicative characters of $\F_q^*$
of order $e$, that is, characters $\chi$
with $\chi^e = \chi_0$, where $\chi_0$ is the principal character.
We also define $\cX_e^* = \cX_e \setminus \{\chi_0\}$.

We recall the following special case of the Weil bound of character sums
(see~\cite[Theorem~11.23]{IwKow}).

\begin{lemma}
\label{lem:Weil}
For any polynomial $Q(X) \in \F_q[X]$ with  $Z$ distinct zeros
in  $\bFq$ and
which is not a perfect $e$-th power in   the ring of polynomials
over $\bFq$,
 and  $\chi \in \cX_e^* $, we have
$$
\left| \sum_{a \in \F_q}\chi\(Q(a)\)\right| \le Z q^{1/2}.
$$
\end{lemma}

We are now ready to establish a lower bound on $N_d(q)$.
\begin{theorem}
\label{thm:Nd bound L}
Suppose that $\gcd(d-1,q)=1$. Then, for any $d\ge2$ and $e=\gcd(d,q-1) \ge 2$, we have
$$
N_d(q) \ge q^{\rho_{d,e} + o(1)}
$$
as $q \to \infty$, where
$$
\rho_{d,e}= \frac{1}{2(e- 1 + \log d/\log e)}.
$$
\end{theorem}

\begin{proof}  We define $J$ by the inequalities
$$
(de^{e-1})^{J+1} \le q^{1/2} /\log q <(de^{e-1})^{J+2} .
$$
Note that for the fixed $d$,  we  have $J\ge 2$ when $q$ is sufficiently large.
For each $j =2, \ldots, J$ and $\gamma \in \tGf$ we
choose a representative $\sigma_{j,\gamma}$  of the quotient
group $\F_q^*/\Delta_e$,
where $\Delta_e$ is the group of $e$-th powers,
and consider the collection
$$
\bsigma = \{\sigma_{j,\gamma}~:~ j =2, \ldots, J, \ \gamma \in \tGf\}.
$$
Let $A(\bsigma)$  denote  the number of $a \in \F_q^*$
such that
$$
\gamma f_a^{(j)}(a)-a \in \sigma_{j,\gamma}  \Delta_e,
\qquad \textrm{for all $j =2, \ldots, J, \ \gamma \in \tGf$}.
$$
Clearly, if for any $\bsigma$ as in the above we have
\begin{equation}
\label{eq:A pos}
A(\bsigma) > 0,
\end{equation}
then the vector~\eqref{eq:vect} takes all $e^{J-1}$ possible values and
thus we have
$$
N_d(q) \ge e^{J-1},
$$
which implies the desired result. So, it remains to prove~\eqref{eq:A pos} for sufficiently large $q$ and the above choice of $J$.

Let $\chi$ be a primitive character of order $e$
(that is, a generator of $\cX_e$).
We can now express $A(\bsigma)$ via the following character sums
$$
A(\bsigma)  =  \sum_{a \in \F_q^*} \frac{1}{e^{(e-1)(J-1)}} \prod_{j=2}^J\,
\prod_{\gamma \in \tGf} \,  \sum_{\alpha_{j,\gamma}=0}^{e-1}
\chi^{\alpha_{j,\gamma}}\(G_{j,\gamma}(a)/\sigma_{j,\gamma}\),
$$
which follows directly from $G_{j,\gamma}(a)=\gamma f_a^{(j)}(a)-a$ and the following orthogonality relations of characters by noticing that $\chi$ is of order $e$ and is a generator of $\cX_e$
(see, for example the orthogonality relations in~\cite[Section~3.1]{IwKow}):
\begin{equation}
\sum_{\alpha_{j,\gamma}=0}^{e-1}
\chi^{\alpha_{j,\gamma}}\(G_{j,\gamma}(a)/\sigma_{j,\gamma}\)
= \left\{\begin{array}{ll}
e & \text{if $G_{j,\gamma}(a)/\sigma_{j,\gamma}\in \Delta_e$,}\\
0 &\text{otherwise.}
\end{array}
\right.
\notag
\end{equation}
Expanding the product, and changing the order of summation we obtain
$e^{(e-1)(J-1)}$ character sums parametrized by different
choices of $\alpha_{j,\gamma}\in \{0, \ldots, e-1\}$,
$j = 2, \ldots, J$, $\gamma \in \tGf$.

Note that  $\chi^{0}\(G_{j,\gamma}(a)/\sigma_{j,\gamma}\)=1$ if and only if $G_{j,\gamma}(a)\ne 0$, and each polynomial $G_{j,\gamma}$ has degree  $d^j$. So,
by estimating the number of distinct zeros of the
polynomial $\prod_{j=2}^J\prod_{\gamma \in \tGf} G_{j,\gamma}$,
the term corresponding to  the choice
$\alpha_{j,\gamma}=0$, $j = 2, \ldots, J$, and $\gamma \in \tGf$,
is not less than $q - 1-(e-1)d^{J+1}$.

For the other terms, using
Lemma~\ref{lem:sqr} (if $d=2$) and Lemma~\ref{lem:e pow}  (if $d\ge 3$),
we apply Lemma~\ref{lem:Weil} to each of them by noticing that each corresponding polynomial has at most $(e-1)d^{J+1}$ distinct zeros. Hence, we obtain
\begin{align*}
A(\bsigma) &\ge \frac{q-1-(e-1)d^{J+1}-(e^{(e-1)(J-1)}-1)(e-1)d^{J+1} q^{1/2}}{e^{(e-1)(J-1)}} \\
&\ge  \frac{1}{e^{(e-1)(J-1)}}\(q-1-(de^{e-1})^{J+1}q^{1/2}\),
\end{align*}
which implies~\eqref{eq:A pos} for sufficiently large $q$ and the above choice of $J$.
\end{proof}

We remark that
$$
\max_{d,\, e\mid d} \rho_{d,e} = \rho_{2,2} = 1/4.
$$

\section{Isomorphism testing of functional graphs}
\label{sec:alg}

\subsection{Preliminaries}

In this section, we give a practical and efficient isomorphism testing algorithm for quadratic polynomials that is linear (in time and memory). We also extend this isomorphism testing algorithm from quadratic polynomials to any arbitrary function with only a slight increase
in the time complexity.

We note that the class of functional graphs over $\F_q$ coincides with the class of directed graphs
on $q$ nodes with all out-degrees equal to 1. However, the in-degrees depend on the particular
function $f$ associated with this graph.

Our algorithms do not depend on the arithmetic structure of $q$ (for example, 
that this is a prime power) 
or on  algebraic properties of the underlying domain (for example, that it has a
structure of a field). Hence we present them for functional graphs over an 
arbitrary set of $n$ elements. 

Clearly, any functional graph is extremely  sparse (with exactly $n$ arcs) and the size of the input that is to be considered for efficient isomorphism testing is linear in the size of an adjacency list (that is, $O(n\log n)$), rather than an adjacency matrix (that is, $O(n^2)$).
We first introduce several graph related notations.

\subsection{Notations and graph input size}

Given two functions $f$ and $h$, we denote the functional graph $\cG_f$ of $f$
as $\cG$ and the functional graph $\cG_h$ of $h$ as $\cH$.
Given a functional
graph $\cG$, we collect its connected components of the  same size
in the  sets $C^{\cG}_{i}$
 with $ 1 \leq i \leq s^{\cG}$,
where $s^{\cG}$ is the total number of distinct sizes of components of $\cG$.
  For each set $C^{\cG}_{i}$ we denote the size of the components in the set by $k^{\cG}_{i}$ and the size of the set itself by $c^{\cG}_{i} = \# C^{\cG}_{i}$.
Let
  \begin{equation}
  \label{eq:kmaxcmax}
k^{\cG}_{*} = \max_{1 \leq i \leq s^{\cG}} k^{\cG}_{i} \mand
c^{\cG}_{*} = \max_{1 \leq i \leq s^{\cG}} c^{\cG}_{i}.
\end{equation}
When there is no ambiguity, we  omit the superscript $\cG$. For convenience we denote the \textit{in-degree} of a vertex $v$ as $d^{-}(v)$ and the corresponding \textit{in-neighbourhood} as $N^{-}(v)$. Since the \textit{out-degree} of any vertex is $1$,
each connected component $C$ in a functional graph has exactly one cycle (which may be a self-loop), which we denote $\cyc(C)$. Each vertex is the root
of a (possibly empty) tree.

\subsection{Isomorphism testing of functional graphs of quadratic po\-lynomials}
\label{eq:IsomQuadr}

We now present our meta-algorithm to test the isomorphism. It comprises three phases:
\begin{description}
\item[Phase 1] Given two functional graphs $\cG$ and $\cH$, we first identify the connected components in each graph, and  the associated cycle and trees in each component.
\item[Phase 2] For each component we produce a canonical encoding.
\item[Phase 3] Finally we construct a prefix tree (formally a trie~\cite{Knuth73}), using the encodings of $\cG$, noting at each vertex of the trie the number of code strings that terminate at that vertex. Then for each encoded component of $\cH$ we match
 the code string against the trie, and decrement the counter at the appropriate trie vertex.
\end{description}

If all counters are zero after this is complete, the two graphs are isomorphic.

The first phase is achieved by combining a cycle detection algorithm and
depth-first search, as laid out in Algorithm~\ref{alg:components}.

\begin{algorithm}
\caption{\textsc{Identification of Connected Components}}
\label{alg:components}
\begin{algorithmic}[1]
\WHILE{unassigned vertices remain}
\STATE Pick an unassigned vertex $v$.
\STATE Perform Floyd's cycle detection algorithm starting at $v$.
\FOR{each cycle vertex $u$}
\STATE Perform a depth-first search on the tree attached at $u$.
\ENDFOR
\ENDWHILE
\end{algorithmic}
\end{algorithm}

The cycle detection algorithm can be done in linear time and space (in the size of each connected component) with Floyd's algorithm~\cite{Knuth69} using only two pointers. The depth-first search is a simple pre-order traversal of the tree and thus only
requires linear time and space~\cite{Knuth68}. In total, the complexity of the first phase is thus linear in time and space with the size of the graph. Note that this phase is independent of the function $f$ (it has linear complexity for any function $f$), 
leading to the following lemma.

\begin{lemma}
\label{lemma:cycle_detection}
For any functional graph $\cG$ of $n$ vertices, Algorithm~\ref{alg:components} identifies all Connected Components and has linear time and memory complexities.\end{lemma}

On the other hand, the second phase depends on the nature of the function. In this section, we focus on quadratic polynomials which provide an especially interesting case when considering the isomorphism of functional graphs. A $k$-ary tree is
\textit{full} (or \textit{proper}) if every non-leaf vertex has exactly $k$ children. (Note that, here, a full $k$-ary tree does not imply that all root-leaf paths have the same length.) As a quadratic polynomial can have at most one repeated root, the functional graph is almost a full binary tree. This allows certain savings in building a canonical labelling of the graph. We note that
 if there is a repeated root, we can deal with the containing component specially by noting which vertex has one child, and adding a dummy second child, then in the two graphs under consideration the dummy vertices must be matched to each other
in any isomorphism.

We recall that the number of different binary trees on $n$ nodes is the $n$-th Catalan number, which for large $n$ is of size $4^n/n^{3/2}$. It is well-known that binary trees can be encoded with exactly $2n+1$ bits~\cite{Knuth68}: by first extending the
original tree  by adding ``special" nodes whenever a null subtree is present (two for leaves and one for non-full internal nodes), and then doing  a pre-order traversal of the tree labelling original nodes with ones and special nodes with zeroes.
When the binary tree is full (which is our case with quadratic polynomials), only $n$ bits suffice to encode the original tree by using a similar technique (simply making the original leaves the special nodes and then performing a pre-order traversal of the tree labelling internal nodes with ones and leaves with zeroes). Our canonical labelling extends this bound by including the cycle with a minimal number of extra bits.

To produce the canonical labelling of a functional graph derived from a quadratic polynomial we employ Algorithms~\ref{alg:general_labelling} and~\ref{alg:LABEL},
 where $\varepsilon$ is the empty string, $s_{i}$ is the string $s$ after circular shift to the bit position $i$, and $\val(s)$ is the interpretation of the string $s$ as a number. In the description of the algorithms we denote string concatenation by
 $\circ$.

\begin{algorithm}
        \caption{\textsc{Canonical Labelling}}
\label{alg:general_labelling}
\begin{algorithmic}
\REQUIRE component $C$
\STATE $\mathrm{s:= \varepsilon}$
\FOR{each vertex $\mathrm{v}$ in $\cyc(\mathrm{C})$}
\STATE $\mathrm{s:=s\circ\text{\textsc{Label}($\mathrm{v}$)}}$
\ENDFOR
\STATE $\mathrm{max := \val(s_{1})}$
\STATE $\mathrm{maxpos := 1}$
\FOR{$\mathrm{i := 2 \text{ to } \Card{\cyc(\mathrm{C})}}$}
\IF{$\mathrm{\val(s_{i}) > max}$}
\STATE $\mathrm{max := \val(s_{i})}$
\STATE $\mathrm{maxpos := i}$
\ENDIF
\ENDFOR
\RETURN $\mathrm{s_{maxpos}}$
\end{algorithmic}
\end{algorithm}

\begin{algorithm}
\caption{\textsc{Label}($v$)}
\label{alg:LABEL}
\begin{algorithmic}[1]
\REQUIRE vertex $\mathrm{v}$.
\IF{$\mathrm{d^{-}(v) = 0}$}
\RETURN ``$\mathrm{0}$''
\ELSE
\STATE $\mathrm{left:=\text{\textsc{Label}} (left(v))}$
\STATE $\mathrm{right:=\text{\textsc{Label}}(right(v))}$
\IF{$\mathrm{left < right}$}
\RETURN $\mathrm{1\circ right \circ left}$
\ELSE
\RETURN $\mathrm{1 \circ left \circ right}$
\ENDIF
\ENDIF
\end{algorithmic}
\end{algorithm}

Algorithm~\ref{alg:general_labelling} runs on each component in turn and produces a canonical label for the component by applying Algorithm~\ref{alg:LABEL}, that is, function {\textsc{Label}($v$)}, to each tree rooted on a vertex of the component's
cycle (Figure~\ref{fig:coded_binary_tree} gives an example), concatenating these labels in the order given by the cycle, then shifting circularly the concatenated label to begin with the cycle vertex that gives the greatest value. Note that if $t$ such
vertices exist (that is, $t$ possible circular shifts leading to the greatest value), the component must have at least a $t$-fold symmetry of rotation around the cycle. Thus, this maximal orientation of the cycle is unique up to automorphism.

Algorithm~\ref{alg:LABEL} encodes a full, rooted, binary tree by assigning each vertex a single bit: $1$ if the vertex is internal, $0$ if it is a leaf. The label is then recursively built by concatenating the assigned bit of the current vertex $v$ to the
 lexicographically sorted labels
of its left child, $left(v)$, and right child, $right(v)$.
In effect this produces a traversal of the tree where we traverse higher weight subtrees first. As each vertex contributes one bit to the label, the total length of the label is $k$ bits for a component of size $k$ and thus $n$ bits for the entire graph.

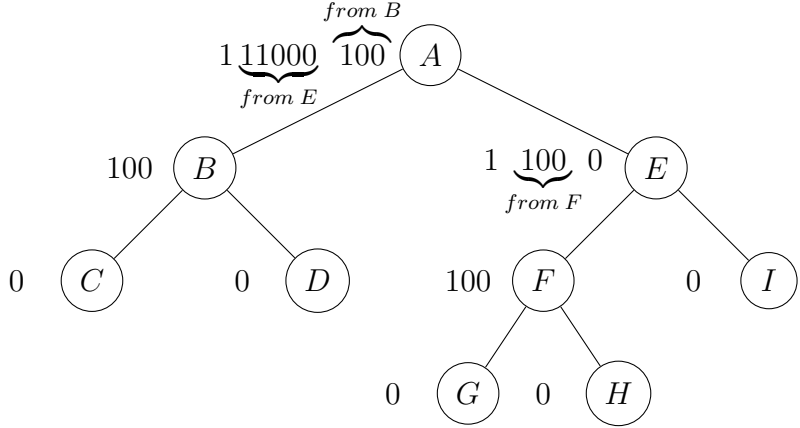
\begin{figure}
\begin{tikzpicture}[level/.style={sibling distance=60mm/#1}]

\node[draw, circle] (A) {$A$}
child{ node[draw, circle] (B) {$B$}
        child{ node[draw, circle] (C) {$C$} }
        child{ node[draw, circle] (D) {$D$} }
}
child { node[draw, circle] (E) {$E$}
        child{ node[draw, circle] (F) {$F$}
                child{ node[draw, circle] (G) {$G$}}
                child{ node[draw, circle] (H) {$H$}}
        }
        child{ node[draw, circle] (I) {$I$}}
};

\node (Alabel) [left of = A, xshift=-6mm] {$1\underbrace{11000}_{from\; E}\overbrace{100}^{from\;B}$};
\node (Blabel) [left of = B] {$100$};
\node (Clabel) [left of = C] {$0$};
\node (Dlabel) [left of = D] {$0$};
\node (Elabel) [left of = E, xshift=-5mm, yshift=-2mm] {$1\underbrace{100}_{from\; F}0$};
\node (Flabel) [left of = F] {$100$};
\node (Glabel) [left of = G] {$0$};
\node (Hlabel) [left of = H] {$0$};
\node (Ilabel) [left of = I] {$0$};
\end{tikzpicture}
\caption{An example binary tree labelled with the canonical coding generated at each level by Algorithm~\ref{alg:LABEL}.}\label{fig:coded_binary_tree}
\end{figure}

\begin{lemma}
\label{lemma:quadratic_canonical_labelling}
For any functional graph $\cG$ of a quadratic polynomial over $\F_q$ with $n=q$ vertices, Algorithms~\ref{alg:general_labelling} and~\ref{alg:LABEL}
build an $n$-bit size canonical labelling  of $\cG$  and have linear time and memory complexities.
\end{lemma}

\begin{proof}
From the description of the traversal process in Algorithms~\ref{alg:general_labelling}
and~\ref{alg:LABEL}, it is clear that each node $v$ in the tree is associated with
a canonical coding \text{\textsc{Label}($v$)} of size $|T_v|$ bits, where $T_v$ is the subtree with root $v$. All leaves are labelled with $0$, and the canonical label
of the whole tree $T_v$ has exactly $k=|T_v|$ bits. The overall memory requirement remains linear: both child labels can be discarded, on the fly, as a parent label is generated.

The worst-case time complexity is slightly more involved, a (lexicographic) sorting is required at each internal node. More precisely, each internal node $v$ requires a number of (lexicographic) bit comparisons $\comp(v)$ equal to the size of the smallest
 label among both children:
\begin{equation}
\label{eq:bincomp}
\begin{split}
\comp(v)  & =  \min\left\{|\text{\textsc{Label}(left($v$))} |,
|\text{\textsc{Label}(right($v$))}|\right\}  \\
 & =  \min\left\{|T_{left}(v)|, |T_{right}(v)|\right\} \leq \fl{\frac{|T(v)|-1}{2}}.
 \end{split}
 \end{equation}
Hence, we see that the worst case for the number of bit comparisons occurs when each subtree is balanced, that is when the full binary tree is complete. Using this simple recurrence, it is easy to see that this leads to less than $n \log n$ bit comparison
s
 for any binary tree of size $n$, which is linear in the size of the input and completes the proof.
\end{proof}

Note that to finally test the isomorphism between two graphs $\cG$ and $\cH$, it remains to compare the canonical labellings of each connected components of each graph with one another (Phase 3). A general brute-force approach (by comparing canonical
labellings of connected components pair-wise) could be ineffective (as shown in the next section). To keep it linear in the size of the input, the third phase builds a trie (or prefix tree) using the encodings of the functional graph $\cG$ by inserting the canonical labelling of each connected components, obtained after Phase~2, one after the other. Each node in the trie is also equipped with a counter initialised to zero and incremented each time the node represents the terminating node of a newly
 inserted canonical labelling of a connected component. It then suffices to check that each canonical labelling of each connected component of $\cH$ is represented in the trie,  decrementing the respective counter each time there is a match. The two functional graphs are isomorphic if there is no mismatch for all canonical labellings of $\cH$ (all counters are zero after all components have been considered), and are non-isomorphic otherwise.

\begin{lemma}
\label{lemma:quadratic_trie}
For any functional graph $\cG$ and $\cH$, each with an $n$-bit canonical labelling, Phase~3 tests their isomorphism by comparing the canonical labelling of $\cG$ and $\cH$ and has linear time and memory complexities.
\end{lemma}

\begin{proof}
It is easy to see that the trie built for the functional graph $\cG$ has at most $n$ nodes. This case is only possible if  all canonical labellings of connected components are disjoint (that is, generate disjoint branches in the tree). As more canonical
 labels overlap, fewer nodes are created. If the labels match, the respective counter (and its size) are incremented, but the cost of increasing the counter remains lower than the cost of creating a distinct branch in the trie. Thus, the overall size
 remains $O(n)$ in memory space. It is also easy to see that creating the initial trie with the canonical labels of $\cG$ takes $O(n)$ time and memory, and the same cost occurs for matching all canonical labels of $\cH$ (and may stop before if the two graphs are not isomorphic).
\end{proof}

Again it is interesting to note that the complexity of Phase~3 does not depend on the type of functional graph but depends solely on the size of the canonical labelling.

Combining Lemmas~\ref{lemma:cycle_detection}, \ref{lemma:quadratic_canonical_labelling} and~\ref{lemma:quadratic_trie}, we obtain the following theorem.

\begin{theorem}
\label{theorem:quadratic}
For any functional graphs $\cG$ and $\cH$ of quadratic functions with $n$ vertices, Phases~1, 2 and~3 combined provide an isomorphism test  that has linear time and memory complexities.
\end{theorem}

It is also interesting to note that the trie built in Phase~3 provides a canonical representation of size $O(n)$ for any functional graph of size $n$. We
 exploit this property to present an algorithm to enumerate all functional graphs corresponding to polynomials of
degree $d$ over $\F_q$  in Section~\ref{sub:counting}.

\subsection{General functional graph isomorphism}

In a general setting, there are numerous standard results that can be used.
The graph isomorphism problem can be solved in time linear in the number of vertices for connected planar graphs~\cite{HW74} and (rooted) trees~\cite{Kelly57}.
Also the fact that our graphs are directed is not an issue as there is a linear-time reduction from directed graph isomorphism to undirected graph isomorphism~\cite{Miller}.

We first give a simple example of how these techniques can be combined to prove a simple upper bound for Functional Graph Isomorphism for arbitrary functions.
Other combinations of these standard techniques are possible to give similar near-linear time.
However we  prove in the remainder of this section that simple and practical, yet efficient (in time and memory), techniques  can be used by extending the algorithms of Section~\ref{eq:IsomQuadr} to arbitrary functions.

\begin{theorem}
\label{thm:standardAlgos}
For any functional graphs $\cG$ and $\cH$ of arbitrary functions with $n$ vertices, there is an isomorphism test using standard algorithms with   $O(c_* n)$ time complexity,
where $c_* = \max\{c^{\cG}_{*}, c^{\cH}_{*}\}$ and $c^{\cG}_{*}$, $c^{\cH}_{*}$ 
are defined by~\eqref{eq:kmaxcmax}.
\end{theorem}

\begin{proof}
A simple approach that can be applied to functional graphs is to run Algorithm~1
(that builds each connected component) and then compare the connected components of the two graphs pairwise, using the appropriate algorithm as a subroutine (for components with a cycle, we can use the planar graph algorithm, for components with a self-loop, we can use 
 the rooted tree algorithm where we treat the vertex with the self-loop as the root). This involves at most $\binom{n}{2}$ comparisons and thus gives an $O(n^2)$ algorithm overall.

Using the sizes of the various components, we can refine this analysis slightly.
Given two functional graphs $\cG$ and $\cH$, if we have the isomorphism
$\cG \cong \cH$ then $s^{\cG} = s^{\cH}$ and for all $i \in [1,s^{\cG}]$ we also have $c^{\cG}_{i} = c^{\cH}_{i}$ and $k^{\cG}_{i} = k^{\cH}_{i}$.  (If the graphs are isomorphic, $c_* = c^{\cG}_{*} =  c^{\cH}_{*}$.) On the other hand,
if one of these pairs of values disagree then $\cG \not \cong \cH$.
Then, denoting these common values as $s$ and $c_i$, $k_i$, $1 \le i \le s$,
clearly for both graphs we have
$$
\sum_{i =1}^{s}c_{i}k_{i} = n,
$$
where $n$ is the order of the graphs.

Clearly we only need to compare components in the same size class. This gives a running time proportional to:
\begin{align*}
\sum_{i=1}^{s} c_{i}^{2}k_{i} \leq c_* \sum_{i=1}^{s} c_{i}k_{i} = c_*n,
\end{align*}
where $c_* = \max_{i=1, \ldots, s} c_i$.
\end{proof}

If each size class $C_{i}$ is bounded, then this na\"{i}ve algorithm is linear in the number of vertices. In the general case however it is likely there are numerous components of the same size~\cite{FlOdl} thus possibly leading to a worst-case bound of $
O(n^2)$ time.
Fortunately  even in this case, as we  now show that we can still solve the isomorphism problem with linear memory complexity and by increasing slightly the cost of building the canonical labels.

The challenge is that, in the general case, we cannot assume that the trees associated with each component are full, nor necessarily have any particular bound on the number of children (note that polynomials of degree $d$ do however have at most $d$ children in the trees).

For the general case we replace Algorithm~\ref{alg:LABEL} with Algorithms~\ref{alg:llabel} and~\ref{alg:rlabel}, and replace the call to \textsc{Label} in Algorithm~\ref{alg:general_labelling} with a call to \textsc{LeftLabel} with the root vertex of the
tree.

\begin{algorithm}
\caption{\textsc{LeftLabel}}
\label{alg:llabel}
\begin{algorithmic}[1]
\REQUIRE vertex $v$
\STATE $\mathrm{label}_{v} := \varepsilon$
\STATE $\mathrm{labelSet}:=\emptyset$
\STATE $\mathrm{finalLabel}:=\varepsilon$
\IF{$v$ is not a leaf}
\STATE $\mathrm{label}_{v} := 1\circ\text{\textsc{LeftLabel($left(v)$)}}$
\IF{$v$ has a right child}
\STATE $\mathrm{labelSet} := \text{\textsc{RightLabel($right(v)$)}}$
\ENDIF
\STATE $\mathrm{labelSet} := \mathrm{labelSet}\cup\{\mathrm{label}_{v}\}$
\STATE \textsc{Sort($\mathrm{labelSet}$)}
\FOR{$i := 1 \text{ to } \Card{\mathrm{labelSet}}$}
\STATE $\mathrm{finalLabel}:= \mathrm{finalLabel}\circ\mathrm{labelSet}[i]$
\ENDFOR
\ELSE
\STATE $\mathrm{finalLabel} := 0$
\ENDIF
\RETURN $\mathrm{finalLabel}$
\end{algorithmic}
\end{algorithm}

\begin{algorithm}
\caption{\textsc{RightLabel}}
\label{alg:rlabel}
\begin{algorithmic}[1]
\REQUIRE vertex $v$
\STATE $\mathrm{label}_{v} := \varepsilon$
\STATE $\mathrm{labelSet}:=\emptyset$
\IF{$v$ is not a leaf}
\STATE $\mathrm{label}_{v} := 1\circ\text{\textsc{LeftLabel($left(v)$)}}$
\ELSE
\STATE $\mathrm{label}_{v} := 0$
\ENDIF
\IF{$v$ has a right child}
\STATE $\mathrm{labelSet} := \text{\textsc{RightLabel($right(v)$)}}$
\ENDIF
\RETURN $\mathrm{labelSet}\cup\{\mathrm{label}_{v}\}$
\end{algorithmic}
\end{algorithm}

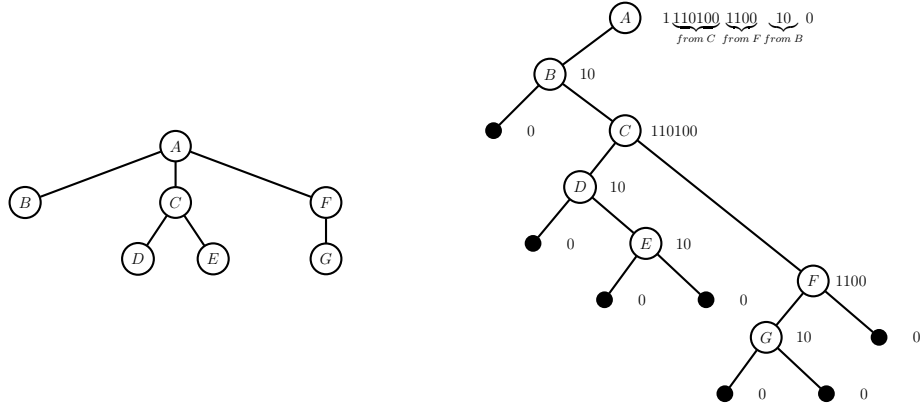
\begin{figure}
\begin{minipage}{.49\textwidth}
\begin{tikzpicture}[level/.style={sibling distance=40mm/#1},thick,scale=0.5, every node/.style={transform shape}]

\node[draw, circle] (A) {$A$}
child{node[draw, circle] (B) {$B$}}
child{node[draw, circle] (C) {$C$}
        child{node[draw, circle] (D) {$D$}}
        child{node[draw, circle] (E) {$E$}}
}
child{node[draw, circle] (F) {$F$}
        child{node[draw, circle] (G) {$G$}}
};
\end{tikzpicture}
\end{minipage}
\begin{minipage}{.49\textwidth}
\begin{tikzpicture}[level/.style={sibling distance=60mm/#1},thick,scale=0.5, every node/.style={transform shape}]
\node[draw, circle] (A') {$A$}
child{node[draw, circle] (B) [xshift=10mm] {$B$}
        child{node[draw, circle, fill] (d1) {}}
        child{node[draw, circle] (C) [xshift=5mm] {$C$}
                child{node[draw, circle] (D) [xshift=-2mm] {$D$}
                        child{node[draw, circle, fill, xshift=-5mm] (d2) {}}
                        child{node[draw, circle] (E) [xshift=10mm] {$E$}
                                child{node[draw, circle, fill, xshift=-5mm] (d3) {}}
                                child{node[draw, circle, fill, xshift=10mm] (d4) {}}
                        }
                }
                child{node[draw, circle] (F) [xshift=40mm, yshift=-25mm] {$F$}
                        child{node[draw, circle] (G) [xshift=-5mm] {$G$}
                                child{node[draw, circle, fill, xshift=-5mm] (d5) {}}
                                child{node[draw, circle, fill, xshift=10mm] (d6) {}}
                        }
                        child{node[draw, circle, fill, xshift=10mm] (d7) {}}
                }
        }
}
child[missing]{node[draw=none, circle, fill=none, xshift=-7mm] (d8) {}};

\node (Alabel) [right of = A, xshift = 20mm, yshift = -3mm] {$1\underbrace{110100}_{from\;C}\underbrace{1100}_{from\;F}\underbrace{10}_{from\;B}0$};
\node (Blabel) [right of = B] {$10$};
\node (Clabel) [right of = C, xshift=3mm] {$110100$};
\node (Dlabel) [right of = D] {$10$};
\node (Elabel) [right of = E] {$10$};
\node (Flabel) [right of = F] {$1100$};
\node (Glabel) [right of = G] {$10$};

\node (d1label) [right of = d1] {$0$};
\node (d2label) [right of = d2] {$0$};
\node (d3label) [right of = d3] {$0$};
\node (d4label) [right of = d4] {$0$};
\node (d5label) [right of = d5] {$0$};
\node (d6label) [right of = d6] {$0$};
\node (d7label) [right of = d7] {$0$};
\end{tikzpicture}
\end{minipage}
\caption{An example non-binary tree (left) and the equivalent binary tree (right) labelled with the canonical coding generated at each level by Algorithms~\ref{alg:llabel} and~\ref{alg:rlabel}. The black vertices in the binary tree on the right are the ad
ded vertices.}\label{fig:general_tree_coded}
\end{figure}

That is, the second phase, in the general case, is achieved by Algorithms~\ref{alg:general_labelling}, \ref{alg:llabel} and~\ref{alg:rlabel}, which take each component of the input graph(s), produce a canonical label by first labelling each tree rooted at
 a cycle vertex, concatenating these labels then shifting the label to obtain the maximum value. Ultimately, we consider these labels as bit strings with the final label of a component taking $2k$ bits where $k$ is the number of vertices in the component.
 We can then encode the graph as a whole with $2n$ bits. To obtain this bound we represent the trees attached to the cycles with left-child-right-sibling binary trees (e.g.~see Knuth~\cite{Knuth68} for binary representation of trees), in which
the right child of a vertex is a sibling and the left child is the first child (we can take any ordering for our purposes).

The two tree labelling algorithms (\textsc{LeftLabel} and \textsc{RightLabel}) together produce the canonical labelling of the tree in several steps. First the tree is implicitly extended to a full binary tree by adding leaf vertices whenever a vertex is
missing a child, except at the root, as it cannot have siblings, so the terminating leaf is superfluous. Each internal vertex is labelled with ``1'' and each leaf with ``0''. Each vertex extends its label by concatenating its label with the label of its
left subtree, then adding this label to the set of labels received from its right
subtree. If a vertex is a left child (that is, it is the first child of its parent in the normal representation), it sorts this set of labels, largest to smallest, concatenates them and passes this label to its parent (Figure~\ref{fig:general_tree_coded} illustrates the process).

\begin{lemma}
\label{lem:Canon Label}
The combined Algorithms~\ref{alg:llabel} and~\ref{alg:rlabel} perform at most $O(k^2)$ bit comparisons and use linear memory space to build a canonical label of size $2k$ bits for any component of size $k$.
\end{lemma}

\begin{proof}
We only need to prove that these two algorithms perform at most $O(k^2)$ bit comparisons. Notice that the main cost at each internal node is to lexicographically sort the labels of its children, and the lexicographic sort of $m$ labels of size $n$ bits costs $O(mn)$ bit comparisons. Fix an arbitrary component $C$ of size $k$. Suppose that there are $t$ trees rooted at the cycle of $C$ with sizes $d_i$, $1\le i\le t$. Then, for labelling $C$, the number of bit comparisons is proportional to
$$
\sum\limits_{i=1}^{t}d_i\cdot k=k\sum\limits_{i=1}^{t}d_i=k^2,
$$
which concludes the proof.
\end{proof}

Combining the costs of labelling for all components, with the rest of the meta-algorithm, we obtain the following result for testing isomorphism.

\begin{theorem}
\label{theorem:general}
For any functional graphs $\cG$ and $\cH$ of arbitrary functions with $n$ vertices, there is an isomorphism test using
$O(k_* \cdot n)$ bit comparisons and linear memory complexity, where $k_* = \max\{k^{\cG}_{*}, k^{\cH}_{*}\}$ and $k^{\cG}_{*}$, $k^{\cH}_{*}$ are defined by~\eqref{eq:kmaxcmax}.
\end{theorem}

\begin{proof}
We need to label all components. Using the sizes of various classes of components in the graph,
that is, Lemma~\ref{lem:Canon Label}, the overall running time is proportional to:
\begin{align}
\label{eq: ciki2}
\sum_{i=1}^{s} c_{i}k_{i}^2 \leq k_* \sum_{i=1}^{s} c_{i}k_{i} = k_*n,
\end{align}
where
$$
k_* = \max_{i=1, \ldots, s} |k_i|.
$$
Combining this result with Lemmas~\ref{lemma:cycle_detection} and~\ref{lemma:quadratic_trie} completes the proof.
\end{proof}

It is interesting to note the trade-off between $c_*$, the maximum number of components of same size
used in Theorem~\ref{thm:standardAlgos} and $k_*$, the largest component,
used in Theorem~\ref{theorem:general}, as it seemingly provides a choice
among algorithms to test the isomorphism depending of related features of
the graph. However, it should be emphasized that the comparison is not
straightforward as the algorithm of Theorem~\ref{theorem:general}
considers bit comparisons as the metric of the time cost, while
Theorem~\ref{thm:standardAlgos} employs more involved algorithms.

We note that the bound~\eqref{eq: ciki2} used in the proof of
Theorem~\ref{theorem:general}
together with  Lemma~\ref{lem:Canon Label} also lead
to an upper bound on the size of the labelling of any functional
graph.

\begin{cor}
\label{cor:Canon Label}
The meta-algorithm used for isomorphism testing in Theorem~\ref{theorem:general} uses at most $O(k_* \cdot n)$  bit comparisons and  linear memory space to build canonical labels of size of $2n$ bits that can be represented in a trie of size $O(n)
$ for any functional graph of size $n$, where $k_{*}$ is defined by~\eqref{eq:kmaxcmax}. 
\end{cor}

\begin{proof}
As each connected component of $k$ vertices contributes $2k$ bits to the final labelling of the graph of size $n$, the total number of bits for representing all components is $2n$. 
Finally, using Phase 3, we can built a trie of at most $2n$ nodes to encode all canonical encodings.
\end{proof}

\subsection{Counting functional graphs}
\label{sub:counting}

We now present an algorithm to enumerate all functional graphs corresponding to
polynomials of degree $d$ over $\F_q$ except that $d=2$ and $2\mid q$.

\begin{theorem}
\label{thm:Nd alg}
For any $d$ and $q$  except for $d=2$ and $2\mid q$, we can create a list of all  $N_d(q)$
distinct functional graphs generated by all degree $d$ polynomials
$f \in \F_q[X]$ in
$O(d^2q^{d}\log^{2}q)$ arithmetic operations and comparisons of
bit strings of length $O(q^2)$.
\end{theorem}

\begin{proof}
Let $m = \gcd(d-1, q-1)$ and let $\Omega=\{\omega_1, \ldots, \omega_m\}$ be a set of representatives of the quotient group $\F_q^*/\cH_m$, where $\cH_m$ is
the group of $m$-th powers in $\F_q$, that is
$$
\cH_m = \{\eta^m:\eta\in \F_q^*\}.
$$

Recall the automorphism $\phi_{\lambda,\mu}$ defined in~\eqref{eq:phi} for any $\lambda\in \F_q^*$ and $\mu \in \F_q$.
We verify that for a polynomial
\begin{equation}
\label{eq: poly f}
f(X) = \sum_{j=0}^d a_j X^j \in \F_q[X], \qquad \deg f = d,
\end{equation}
 we have
\begin{align*}
\phi_{\lambda,\mu}^{-1} \circ f \circ \phi_{\lambda,\mu}(X)
&= \lambda^{-1} \(f(\lambda X + \mu)  - \mu\) \\
&=\sum_{j=0}^d A_j X^j,
\end{align*}
for some coefficient $A_j \in \F_q$, $j=0, \ldots,d$. In particular, we have
\begin{equation}
\label{eq:AA}
\begin{split}
&A_{d} = \lambda^{d-1} a_d, \\
&A_{d-1} = \lambda^{d-2} \mu da_d +
\lambda^{d-2} a_{d-1},\\
&A_{d-2}=\lambda^{d-3} \mu^2\frac{d(d-1)}{2}a_d+\lambda^{d-3}\mu(d-1)a_{d-1}+\lambda^{d-3}a_{d-2} .
\end{split}
\end{equation}

We claim that for any polynomial
$f \in \F_q[X]$ of  the form~\eqref{eq: poly f} we can find $\lambda \in \F_q^*$
such that $A_{d}(a_d;\lambda,\mu)  \in \Omega$. Indeed, we can assume that $a_d=\omega_i\eta^m$ for some $i$ ($1\le i \le m$) and $\eta\in \F_q^*$.
Since there exist two integers $s,t$ such that $s(d-1)+t(q-1)=m$, we have $a_d=\omega_i\eta^{s(d-1)}$. Then choosing $\lambda=\eta^{-s}$, we get $A_{d}(a_d;\lambda,\mu)=\omega_i  \in \Omega$.

If $\gcd(d,q)=1$, then we can find an element $\mu \in \F_q$
such that  $A_{d-1}= 0$. Thus, it suffices to consider the
polynomial $F$ of the form $F(X) = A_dX^d + g(X)$ where $A_d \in \Omega$
and $g(X) \in \F_q[X]$ is of  degree $d-2$.
Therefore, it is enough to examine the graphs $\cG_F$ only for such $mq^{d-1}< dq^{d-1}$ polynomials $F$.

Assume that $\gcd(d,q)\ne 1$. Then, we must have $d>2$. Noticing that $d(d-1)/2$ is divisible by the characteristic $p$ of $\F_q$, if $A_{d-1}\ne 0$ (that is $a_{d-1}\ne 0$), we can choose $\mu \in \F_q$ such that $A_{d-2}=0$. So, it is enough to examine the graphs $\cG_F$ only for such $2mq^{d-1}< 2dq^{d-1}$ polynomials $F$ (that is satisfying $A_d\in \Omega$ together with $A_{d-1}=0$ or $A_{d-2}=0$).

Given  such a polynomial $f \in \F_{q}[X]$ of degree $d$,
we can construct the graph $\cG_{f}$ in time $O(dq\log^{2} q)$ (see~\cite{BZ10}). After this, by Corollary~\ref{cor:Canon Label}, for each graph, in time
$O(q^2)$ we compute its canonical label. Using the above discussion and inserting these labels in an ordered list
of length at most $ N_{d}(q)$ (or discarding if the label already in the list) gives an overall time of $O(dq^{d-1}\cdot dq\log^{2} q)=  O(d^2q^{d}\log^{2}q)$.
 \end{proof}

 In particular, the running time of the algorithm of
 Theorem~\ref{thm:Nd alg} is at most $d^2q^{d+2+o(1)}$.

\section{Numerical results}

\subsection{Preliminaries}

We note that the periodic structure of functional graphs has been extensively studied
numerically (see, for example,~\cite{BGHKST}).
These results indicate that ``generic'' polynomials lead to graphs with cycle
lengths with the same distribution as of those associated with random maps
(see~\cite[Section~5]{BGHKST}).

It is not difficult to see that for an odd $q$, the functional graph of any quadratic
polynomial over $\F_q$ has $(q-1)/2$ leaves. Indeed, for $f(X)= X^2+a $ the node $a$ is always an
inner node with in-degree $1$ while other nodes are of in-degree 0 or 2. Thus there are
$1 + (q-1)/2 = (q+1)/2$ inner nodes and $(q-1)/2$ leaves. On the other hand,  the graph of a random map on $p$ nodes
is expected to have $p/e \approx 0.3679 \ p$ leaves. It is possible that there are some
other structural distinctions. Motivated by this,  we have studied numerically
several other parameters of functional graph.

Our tests have been limited to quadratic polynomials in prime fields, which
can be further limited to polynomials of the form $f(X) = X^2 + a$, $a \in \F_p$.
Various properties of the corresponding function graphs $\cG_f$ have been
tested for all $p$ polynomials of this form for the following sequences of primes:
\begin{itemize}
\item all  odd primes up to 100 (mostly for the purpose of testing our algorithms,
but this has also revealed an interesting property of $N_2(17)$);
\item for the sequence of primes between 101 and 102407
where each prime is approximately twice the size of its
predecessor;
\item for the sequence of 30 consecutive primes between 204803 (which could also be
viewed as the last element of the previous group)  and 205171;
\item for the sequence of 10 consecutive primes between 500009 and 500167;
\item for the  prime 1000003.
\end{itemize}

For these primes, we  tested the number of
distinct primes and also average and extreme values of
several basic parameters of the graphs
$\cG_f$.

Our numerical results revealed that some  of these parameters are the same as those
of random graphs, but some  (besides the aforementioned number of inner nodes)
deviate in a rather significant way. More importantly (and as far as we are aware), some of these parameters of graphs have never been discussed in the literature
before this work. Using our
practical algorithms
of Section~\ref{sec:alg} we have initiated the study of these interesting parameters.

We present some of our numerical results (limited to those that show some
new and unexpected aspects in the statistics of the graphs $\cG_f$),
only for the primes of the last two groups,
that is, for the set of primes
\begin{equation*}
\begin{split}
\{500009&, 500029, 500041, 500057, 500069, 500083, \\
& 500107, 500111, 500113, 500119, 500153, 500167, 1000003\}.
\end{split}
\end{equation*}

\subsection{Number of distinct graphs}

For all tested primes we have $N_2(p) = p$ except for $p = 2,17$ in which cases
$N_2(2)=3$ and $N_2(17) = 16$. This indicates that most likely we have $N_2(p)=p$ for any
odd prime $p$, except for $p=2,17$. However, proving this may be difficult as
the cases of $p=2,17$ show that there is no intrinsic reason for this to be true
(apart from the fact that, as $p$ grows, the probability for this to occur becomes smaller).

\subsection{Cyclic points and  the giant components}

Our numerical tests show that the   average values of
\begin{itemize}
\item the number of cyclic points,
\item the size of the largest connected components,
\end{itemize}
behave like expected from  random maps, which are  predicted to be $\sqrt{\pi p/2}$, (see~\cite[Theorem~2~(ii)]{FlOdl}) and $\gamma p$
where $\gamma=0.75788\ldots$,
(see~\cite[Theorem~8~(ii)]{FlOdl}), respectively.

It is also interesting to investigate the  extreme values.
More precisely, let $c(f)$ be the number of cyclic points of $\cG_f$
and let
$$
C(p) = \max\{c(f)~:~ f(X) = X^2 + a, \ a\in \F_q\}.
$$
In all our tests, except for the primes $p=5, 13, 17$, the value of $c(f)$
is maximised on the function graphs of polynomials
$f_0(X) = X^2$ and $f_{-2}(X) = X^2-2$, for which
$$
c(f_0) = r+1\mand c(f_{-2}) = (r + s)/2,
$$
where $r$ is the largest odd divisor of $p-1$ and
$s$ is the largest odd divisor of $p+1$,
see~\cite[Theorem~6~(b)]{VaSha} and~\cite[Corollary~18~(b)]{VaSha},
respectively  (note that in~\cite{VaSha} the polynomials are
considered as acting on $\F_p^*$). In particular, if $p\equiv 3 \pmod 4$ then the function
graph of $X^2$ has  the largest possible number
of cyclic points, which is  $(p+1)/2$. Hence,
$$
C(p) = (p+1)/2, \qquad \text{for}\ p\equiv 3 \pmod 4.
$$
 We also note that for any $p\ge 3$,
\begin{equation}
\label{eq:Cp LB}
C(p) \ge \max\{r+1, (r + s)/2\} \ge (p+3)/4.
\end{equation}
Furthermore,  if $f(X) = X^2 + a$ with $a\in \F_p^*$ then the number of cyclic points of
$\cG_f$ is at most $3p/8 + O(1)$. Indeed, let $\cV_f = \{f(x)~:~x\in \F_p\}$ be the value set of $f$
(that is, the set of inner nodes of $\cG_f$). Clearly, $v \in\cV_f$ if $v-a$ is
quadratic residue modulo $p$.  Since for the sums of Legendre symbols
modulo $p$ we
have
$$
\left|\sum_{v \in \F_p} \(\frac{(v-a)(-v-a)}{p} \) \right| = 1
$$
(see~\cite[Theorem~5.48]{LN}), we see that there are $p/4 + O(1)$
values of $v \in \F_p$ with $v, -v \in \cV_f$. However, because $f(v) = f(-v)$,  it is clear that
only one value out of $v$ and $-v$ can be a cyclic point. Hence, the number
of cyclic points in $\cG_f$ for $f(X) = X^2 + a$ with $a\in \F_p^*$ is at most
$3p/8 + O(1)$.
In particular, we now see from~\eqref{eq:Cp LB} that
$$
C(p) = 3p/8 + O(1), \qquad \text{for}\ p\equiv 5 \pmod 8.
$$

The smallest number of cyclic points has achieved the value 2 for all tested primes
except $p=3$ and $p=7$ (for which this is 1).

In Table~\ref{tab:CyclPoint}, we provide some numerical data for
the number of cyclic points taken over all
polynomials except for the above two special polynomials. In particular,
we give the results for
$$
C^*(p) = \max\{c(f)~:~ f(X) = X^2 + a, \ a\in \F_q\setminus\{0,-2\}\}.
$$

\begin{table}[ht]
  \centering
  \begin{tabular}{|r|r|r|r|r|}
\hline
Prime & Min & Max & Average & Expected\\
\hline
500009 &  2 &  3578 &  886.2239149 &  886.2349015\\
500029 &  2 &  3620 &  885.9897086 &  886.2526257\\
500041 &  2 &  3798 &  885.0688786 &  886.2632600\\
500057 &  2 &  3468 &  884.9626481 &  886.2774389\\
500069 &  2 &  3556 &  885.8313906 &  886.2880730\\
500083 &  2 &  3596 &  884.9700189 &  886.3004792\\
500107 &  2 &  3527 &  884.5065536 &  886.3217460\\
500111 &  2 &  3732 &  884.3407057 &  886.3252912\\
500113 &  2 &  3805 &  885.1602624 &  886.3270634\\
500119 &  2 &  3873 &  884.5585953 &  886.3323802\\
500153 &  2 &  3472 &  884.8337362 &  886.3625078\\
500167 &  2 &  3644 &  884.7563204 &  886.3749130\\
1000003 &  2 &  5101 &  1252.451837 & 1253.316017\\
\hline
\end{tabular}
\caption{Cyclic points of  polynomials $f(X) \ne X^2, X^2-2$}
\label{tab:CyclPoint}
\end{table}

It is quite apparent from Table~\ref{tab:CyclPoint}
(and from our results for smaller primes)
that both the maximum values (that is,  $C^*(p)$) and the average values
behave regularly and, as we have mentioned, the average value
fits the model of a random map quite precisely.
We have not attempted to explain the behaviour of  $C^*(p)$.

The size of the largest component achieved the largest possible value $p$
in all tested cases (thus, for any $p$ some quadratic polynomial
generates a graph with just one connected component, see Table~\ref{tab:ConComp}
below). On the other hand, the smallest achieved size of the
 largest component does not seem to have a regular behaviour or even monotonicity.

\subsection{Number of components}

On the other hand, the average number of connected components has exhibited a
consistent (but slowly decreasing) bias of about 9.5\% over the
predicted value   $0.5 \log p$, see~\cite[Theorem~2~(i)]{FlOdl}.

For every tested prime, at least one
graph $\cG_f$ has just 1 component, while the largest number of components
has been behaving quite chaotically in all tested ranges.

The above is illustrated in Table~\ref{tab:ConComp}:

\begin{table}[H]
  \centering
  \begin{tabular}{|r|r|r|r|r|r|}
\hline
Prime & Min & Max & Average & Expected &Ratio\\
\hline
500009 & 1 & 135 & 7.19772 & 6.561190689 & 1.097014298\\
500029 & 1 & 631 & 7.20138 & 6.561210688 & 1.097568778\\
500041 & 1 & 58 & 7.19640 & 6.561222687 & 1.096807766\\
500057 & 1 & 139 & 7.19259 & 6.561238685 & 1.096224409\\
500069 & 1 & 48 & 7.19785 & 6.561250684 & 1.097024081\\
500083 & 1 & 56 & 7.19328 & 6.561264682 & 1.096325228\\
500107 & 1 & 129 & 7.19792 & 6.561288677 & 1.097028397\\
500111 & 1 & 104 & 7.19801 & 6.561292676 & 1.097041445\\
500113 & 1 & 160 & 7.19402 & 6.561294676 & 1.096432999\\
500119 & 1 & 81 & 7.19518 & 6.561300675 & 1.096608791\\
500153 & 1 & 143 & 7.19312 & 6.561334665 & 1.096289150\\
500167 & 1 & 77 & 7.19699 & 6.561348661 & 1.096876629\\
1000003 & 1 & 22 & 7.54330 & 6.907756779 & 1.092004285\\
\hline
\end{tabular}
\caption{Numbers of connected components}
  \label{tab:ConComp}
\end{table}

\subsection{Most popular component size}

As we have mentioned, motivated by the complexity
bounds of the algorithms of Section~\ref{sec:alg}, we
calculated the most popular size of the
connected components of $\cG_f$.
Our results for large primes are given  in Table~\ref{tab:PopSize1}.
For all tested primes $p$, the minimal value of the most common size is 1 or 2 (in fact,
2 becomes more common than 1 as $p$ grows), while the largest value is $p$,
as in  accordance with Table~\ref{tab:ConComp}, for every $p$
there is always connected a graph $\cG_f$.
The average value certainly shows a regular growth. However,
there does not seem to be any results for this parameter for graphs
of random maps, so we have not been able to compare the graphs $\cG_f$
with such graphs. Our numerical results seems to suggest that the average
of the most common size is proportional to $p^{1/2}$. However, we
believe that more numerical experiments are needed before one can confidently
formulate any conjectures.

\begin{table}[H]
  \centering
  \begin{tabular}{|r|r|r|r|r|r|}
\hline
Prime & Min & Max & Average  \\
\hline
500009 & 1 & 500009 & 1689.24\\
500029 & 2 & 500029 & 1642.27\\
500041 & 2 & 500041 & 1604.86\\
500057 & 1 & 500057 & 1670.49\\
500069 & 2 & 500069 & 1638.32\\
500083 & 2 & 500083 & 1628.07\\
500107 & 2 & 500107 & 1635.19\\
500111 & 2 & 500111 & 1657.12\\
500113 & 2 & 500113 & 1655.44\\
500119 & 2 & 500119 & 1573.22\\
500153 & 2 & 500153 & 1690.84\\
500167 & 2 & 500167 & 1638.63\\
1000003 & 2 & 1000003 & 2272.39\\
\hline
\end{tabular}
\caption{Most common size of components}
  \label{tab:PopSize1}
\end{table}

Furthermore, we also computed the number of components
of the most popular size (see Table~\ref{tab:PopSize2}).
Clearly, the minimal value has been 1 for all tested primes
(as before, we appeal to Table~\ref{tab:ConComp} that shows
that  for every $p$  there is connected graph $\cG_f$).
However, the largest multiplicity exhibits a surprising chaotic
behavior.

The average value clearly converges to a certain constant. However, we
made no attempt to conjecture the nature of this constant.

As above with the case of the most common size, this parameter has not been studied
and there is no random map model to compare against our
results.

\begin{table}[H]
  \centering
  \begin{tabular}{|r|r|r|r|r|r|}
\hline
Prime & Min & Max & Average  \\
\hline
500009 & 1 & 75 & 1.18909\\
500029 & 1 & 465 & 1.18856\\
500041 & 1 & 18 & 1.18776\\
500057 & 1 & 104 & 1.18739\\
500069 & 1 & 18 & 1.18811\\
500083 & 1 & 24 & 1.18729\\
500107 & 1 & 56 & 1.18853\\
500111 & 1 & 40 & 1.18767\\
500113 & 1 & 80 & 1.18835\\
500119 & 1 & 24 & 1.18710\\
500153 & 1 & 108 & 1.18818\\
500167 & 1 & 54 & 1.18826\\
1000003 & 1 & 4 & 1.18843\\
\hline
\end{tabular}
\caption{Numbers of components of the most common size}
  \label{tab:PopSize2}
\end{table}

\section{Further Directions}

It is certainly interesting to study multivariate
analogues of our results, that is, to study graphs
on $q^m$ vertices, generated by a system of $m$ polynomials
in $m$ variables over $\F_q$. It is possible that some results and
ideas of~\cite{FuBa} can be useful here.

Polynomial graphs over residue rings are also interesting and
apparently totally unexplored objects of study. They may also
exhibit some new and rather unexpected features.

Finally, we pose an open question of obtaining
reasonable approximations to
the expected values of the quantities
$k^{\cG}_{*}$ and  $c^{\cG}_{*}$ for a graph
associated with a random map.

\section*{Acknowledgement}

The authors are very grateful to Domingo G\'omez-P\'erez and to the referees for many useful comments.

The research of S.~V.~K. was partially supported
by Russian Fund for Basic Research, Grant N.~14-01-00332,
and Program Supporting Leading Scientific Schools, Grant Nsh-3082.2014.1;
that of F.~L.   by a Marcos Moshinsky fellowship;
that of B.~M. by Australian Research Council Grant DP110104560;
that of M.~S. by Australian Research Council Grant
DP130100237;
and that of  I.~E.~S. by Australian Research Council Grant
DP130100237  and Macquarie University Grant MQRDG1465020.

\end{document}